\documentclass[11pt]{amsart}  

\usepackage{amsbsy,amsmath,amsthm,amssymb,color,verbatim,ulem,color}

\usepackage{mathtools}
\usepackage{soul}
\usepackage{cancel,shuffle}
\usepackage{stackrel}
\usepackage{fdsymbol}
\usepackage{multicol}
\usepackage[all]{xy}
\usepackage{xargs}

\usepackage[backrefs]{amsrefs}
\usepackage{float}
\usepackage{bm}
\usepackage{fullpage,cancel,hyperref}
\usepackage{float}
\usepackage{wrapfig}
\usepackage[dvipsnames]{xcolor}
\usepackage{graphicx}
\usepackage{subcaption}
\usepackage{multicol}
\usepackage{multirow}
\usepackage{longtable}
\usepackage{makecell}
\usepackage{supertabular}
\usepackage{tikz}
\tikzstyle{vertex}=[circle, draw, inner sep=0pt, minimum size=4pt]

\tikzstyle{vtx}=[circle, draw, inner sep=0pt, minimum size=8pt]
\usetikzlibrary{decorations.pathreplacing}

\definecolor{darkgreen}{cmyk}{.9,0,.9,.2}
\definecolor{midgray}{gray}{0.60}
\definecolor{lightgray}{gray}{0.90}
\definecolor{lmgray}{gray}{0.70}
\usepackage{colortbl}
\usetikzlibrary{arrows.meta,calc,decorations.markings,math,arrows.meta}

\tikzset{->-/.style={decoration={
  markings,
  mark=at position #1 with {\arrow{>}}},postaction={decorate}}}
  \tikzset{-<-/.style={decoration={
  markings,
  mark=at position #1 with {\arrow{<}}},postaction={decorate}}}

\theoremstyle{theorem} 
\newtheorem{thm}{Theorem}[section] 
\newtheorem{prop}[thm]{Proposition}
\newtheorem{lem}[thm]{Lemma} 
\newtheorem{cor}[thm]{Corollary}

\theoremstyle{definition}                 
\newtheorem{defn}[thm]{Definition}

\theoremstyle{example}                      

\theoremstyle{remark}                      
\newtheorem{exmp}[thm]{Example}  
\newtheorem{rem}[thm]{Remark}

\numberwithin{equation}{section}

\title{Enumerating $k$-Naples Parking Functions\\Through Catalan Objects}
\author{Jo\~ao Pedro Carvalho}
\address[J.~P. Carvalho]{Haverford College, Haverford, PA 19041}
\email{\textcolor{blue}{\href{mailto:jdecarvalh@haverford.edu}{jdecarvalh@haverford.edu}}}
\author{Pamela E. Harris}
\address[P.~E. Harris]{Department of Mathematics and Statistics, Williams College, Williamstown, MA 01267}
\email{\textcolor{blue}{\href{mailto:peh2@williams.edu}{peh2@williams.edu}}}
\author{Gordon Rojas Kirby}
\address[G.~Rojas Kirby]{School of Mathematical and Statistical Sciences, Arizona State University, Tempe, AZ 85281}
\email{\textcolor{blue}{\href{mailto:}{girkirby@asu.edu}}}
\author{Nico Tripeny}
\address[N. Tripeny]{Haverford College, Haverford, PA 19041}
\email{\textcolor{blue}{\href{mailto:ntripeny@haverford.edu}{ntripeny@haverford.edu}}}
\author{Andr\'es~R.~Vindas-Mel\'endez}
\address[A.~R.Vindas-Mel\'endez]{Mathematical Sciences Research Institute \& Department of Mathematics, University of California, Berkeley, CA 94720}
\email{\textcolor{blue}{\href{mailto:avindas@msri.org}{avindas@msri.org}}}
\date{}

\begin{document}

\maketitle

\begin{abstract}
    
    This paper studies a generalization of parking functions named $k$-Naples parking functions, where backward movement is allowed. One consequence of backward movement is that the number of ascending $k$-Naples is not the same as the number of descending $k$-Naples. This paper focuses on generalizing the bijections of ascending parking functions with combinatorial objects enumerated by the Catalan numbers in the setting of both ascending and descending $k$-Naples parking functions. These combinatorial objects include Dyck paths, binary trees, triangulations of polygons, and non-crossing partitions. Using these bijections, we enumerate both ascending and descending $k$-Naples parking functions.
    
\end{abstract}

\section{Introduction}\label{sec:Introduction}
Parking functions are special types of integer sequences that were proposed independently by Ronald Pyke \cite{Pyke} as well as by Alan Konheim and Benjamin Weiss \cite{KonheimAndWeiss} in order to study hashing problems in computer science. 
If we have a sequence of $n$ integers all belonging to $[n]:=\{1,2,\dots, n\}$, we call it a \textbf{parking preference} of length $n$.
A \textbf{parking function} of length $n$ is a special type of parking preference $(a_1,a_2,\dots,a_n)$, that allows $n$ cars $c_1,c_2,\dots,c_n$ with respective preferences $a_1,a_2,\dots,a_n$ to park in a one-way street with $n$ consecutively ordered parking spots according to the following rules:

\begin{enumerate}
    \item $c_1$ parks in its preferred spot;
    \item Every new car parks in its preferred spot if it is not occupied, otherwise it parks in the next available spot.
\end{enumerate}
For example, the parking preference $(2,2,1,4)$ is a parking function of length 4, where $c_1$ parks in the second spot, $c_2$ in the third, $c_3$ in the first, and $c_4$ in the fourth.

Special subsets of parking functions are the monotonic ones, that correspond to \textit{ascending} (weakly increasing) or \textit{descending} (weakly decreasing) parking preferences. 
The set of ascending parking functions of length $n$, as well as the set of descending parking functions of length $n$, are counted by the Catalan numbers.
There are many well-known bijections between either of these subsets of parking functions and a variety of Catalan objects. 
Figure \ref{fig:CatalanObjects} illustrates some of these Catalan objects that correspond to the ascending parking function $(1,2,2,4)$, i.e., the ascending rearrangement of the parking function $(2,2,1,4)$, which include Dyck paths, binary trees, triangulations of $n$-gons, and non-crossing partitions of the set $[n]$.
We remark that the number of ascending and descending parking functions is the same follows from the fact that if a given parking preference is a parking preference, then so are all of its rearrangements.

\begin{figure}[h]
    \centering
\begin{tikzpicture}
\begin{scope}[scale=0.5, shift=(-90:2.5)]
	\draw[blue, very thick] (0,0) to (8,0);
	\draw[very thick] (0,0) to (1,1) to (2,0) to (4,2) to (6,0) to (7,1) to (8,0);

	\fill (0,0) circle (4pt);
	\fill (1,1) circle (4pt);
	\fill (2,0) circle (4pt);
	\fill (3,1) circle (4pt);
	\fill (4,2) circle (4pt);
	\fill (5,1) circle (4pt);
	\fill (6,0) circle (4pt);
	\fill (7,1) circle (4pt);
	\fill (8,0) circle (4pt);
\end{scope}

\begin{scope}[shift=(0:5.5), scale=0.75]
	\draw[very thick] 
		(0,0) to (2,-2)
		(1,-1) to (0,-2)
		;
		
	\fill (0,0) circle (3pt);
	\fill (1,-1) circle (3pt);
	\fill (2,-2) circle (3pt);
	\fill (0,-2) circle (3pt);
	
\end{scope}

\begin{scope}[shift=($(-90:3)+(0:2)$),scale=.5]
	\draw[very thick](0:2) to (60:2) to (120:2) to (180:2) to (-120:2) to (-60:2) to (0:2) to ((60:2);
	
	\draw[line width=1.5pt, red] (60:2) to (120:2);
	
	\draw[very thick](60:1.95) to (180:1.95) to (-60:1.95) to (60:1.95) to (180:1.95)
		;
	
\end{scope}

\begin{scope}[shift=($(-90:3)+(0:6)$), scale=.5]
	\draw[very thick] (1,1) to (-1,-1);		
	
	\fill (1,1) circle	 (4pt) node[above right]{\Large $2$};
	\fill (-1,1) circle	 (4pt) node[above left]{\Large $1$};
	\fill (-1,-1) circle(4pt) node[below left]{\Large $3$};
	\fill (1,-1) circle	 (4pt) node[below right]{\Large $4$};
	
\end{scope}
\end{tikzpicture}
    \caption{Catalan objects corresponding to the parking function $(1,2,2,4)$.}
    \label{fig:CatalanObjects}
\end{figure}
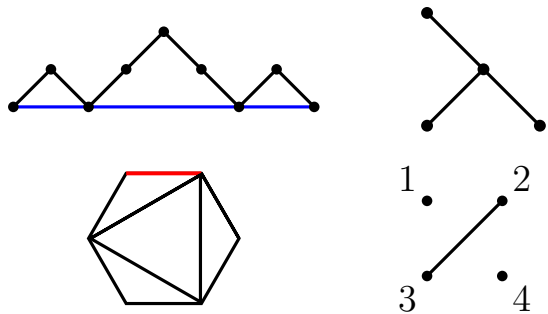

Many generalizations to parking functions were proposed throughout the years, including allowing cars with different lengths and starting with some spots already filled, and a reader interested in exploring all the different directions may find \cite{parkingadventure} a useful resource.  
In this paper, we focus our attention
on the generalization known as \textbf{Naples parking functions}. 
First proposed in \cite{BaumgardnerHonors}, Naples parking functions is different from parking function because cars would go to their preferred spot, and if this spot was already occupied, they then first check one spot before, parking there if available or moving forward if that space is occupied. 
Similarly, a \textbf{$k$-Naples parking function} would allow cars to check up to $k$ spots preceding their parking preference, in decreasing order, before moving forward. 
For example, the parking preference $(6,6,6,5,5,2,1)$ is a $2$-Naples parking function as the cars $c_1,\dots, c_7$ park in positions $6,5,4,3,7,2,1$, respectively, with the second car moving back once and the third and fourth cars moving back twice since both their preferred spot and the one directly behind it are taken. 
We see this is not a $1$-Naples parking function as car $c_5$ cannot park in that setting. 
Also, we remark that we refer to $1$-Naples parking functions simply as Naples parking functions. 
Similarly, $0$-Naples parking functions are just the set of parking functions.
However, unless otherwise specified, we adopt the convention that $k\geq 1$ for all $k$-Naples functions considered in this paper. 

We now summarize our main results in this paper:
\begin{itemize}
    \item As we commented previously, all rearrangements of a parking function are parking functions. In Section \ref{sec:Rearrangements} we answer the analogous question for $k$-Naples parking functions, by establishing that given a parking preference, all of its rearrangements are $k$-Naples parking functions if and only if its ascending rearrangement is a $k$-Naples parking function. This is the statement of Theorem \ref{thm:Rearrangements}.
\item We then restrict our study to ascending and descending $k$-Naples parking functions and in sections \ref{sec:DyckPaths} and \ref{sec:BinaryTrees}, we present bijections between ascending and descending $k$-Naples parking functions and families of Dyck paths and binary trees, respectively. 
\item In Section \ref{sec:Monotonic} we use the bijections found in the previous sections to give formulas to enumerate ascending and descending $k$-Naples parking functions. These results give connections to Fine numbers (\textcolor{blue}{\href{https://oeis.org/A000957}{A000957}}) and convolution of the Catalan numbers with the Fine numbers (\textcolor{blue}{\href{https://oeis.org/A000958}{A000958}}).
\item We also consider bijections between ascending and descending $k$-Naples parking functions and other Catalan objects. This is the content of Section \ref{sec:OtherBijections}. 
\end{itemize}
Following these results, we conclude the article in Section \ref{sec:FutureWorks} by detailing some future directions of research on these topics. 



\section{Rearrangements of $k$-Naples Parking Functions}\label{sec:Rearrangements}

An interesting question regarding $k$-Naples parking functions is if all the rearrangements of $k$-Naples parking functions remain $k$-Naples parking functions, since this is true for traditional parking functions. 
However, this is not the case, as we illustrate in Example \ref{exmp:BadRearrangement}. 
The bijective correspondence between ascending and descending parking functions does not hold in the $k$-Naples case for $k>0$. This motivates us to study both ascending and descending $k$-Naples parking functions separately to understand their similarities and differences in the hopes of achieving a better understanding of general $k$-Naples parking functions.

\begin{exmp}\label{exmp:BadRearrangement}
We have $(6,6,6,5,5,2,1)$ as a descending parking preference of length $7$ and we have $(1,2,5,5,6,6,6)$ as its corresponding ascending parking preference. From above, we see $(6,6,6,5,5,2,1)$ is a descending $2$-Naples parking function. We also can see that $(1,2,5,5,6,6,6)$ is an ascending $3$-Naples parking function, but not an ascending $2$-Naples parking function. 
\end{exmp}


As a starting point to begin exploring exactly when  rearrangements of a $k$-Naples parking function are still $k$-Naples, we define a slightly modified way of writing parking preferences that can capture the position of cars step by step as they park. 
Then, we prove a lemma that expose small modifications in the parking preference that conserve its status as a parking function.

\begin{defn}
An \textit{$i$-filled} parking preference
of length $n$ is an ordered pair of sequences in $[n]$,  of the form $((d_1,\dots,d_i),(a_{i+1},\dots, a_n))$, such that each $d_j$ for $1\leq j\leq i$ represents the spot in which car $c_j$ has already parked, and each $a_j$ for $i<j\leq n$ represents the preference of car $c_j$ that has yet to park. 
If all cars can park using the $k$-Naples rules we call this an \textit{$i$-filled} $k$-Naples parking function.
\end{defn}

\begin{lem}\label{lem:t1}
If $P_1=(a_1, a_2,\ldots,a_n)$ is a $k$-Naples parking function with cars parking in spots $(d_1,\ldots, d_n)$ and we consider the $i$-filled parking preference $P_2=((p_1,\ldots,p_i),(a_{i+1},\ldots,a_n))$ with $p_j=d_j$ for all but exactly one $l\leq i$ where $p_l<d_l$, then $P_2$ is a $i$-filled $k$-Naples parking function.
\end{lem}

\begin{proof}
We prove this by induction on $n-i$. Let $n-i=1$ or $i=n-1$. Suppose that $P_2$ is an $(n-1)$-filled parking preference $((p_1,\dots, p_{n-1}),(a_n))$ so that,  according to this preference, cars $c_1,\dots, c_{n-1}$ park in spots $p_1,\dots p_{n-1}$. By assumption $p_j=d_j$ for all $j\leq i$ except for one $1\leq l\leq n-1$ where $p_l<d_l$. Since all cars must park in distinct spots we must have $d_n=p_l<d_l$ so that $c_n$ spot $d_l$ is unoccupied when $c_n$ goes to park and $c_n$ is able to park in spot $d_l$. Thus, $P_2$ is an $(n-1)$-filled parking preference.

Now, suppose the same is true for every $k<i<n$ and suppose that $P_2$ is an $k$-filled parking preference $((p_1,\dots, p_k),(a_{k+1},\dots a_n))$ so that,  according to this preference, cars $c_1,\dots, c_{k}$ park in spots $p_1,\dots p_{k}$. By assumption $p_j=d_j$ for all but exactly one $l\leq i$, where $p_l<d_l$. Thus, $p_l=d_j$ for some $j>k$. Now consider where car $c_{k+1}$ parks according to $P_2$. 

If $p_l=d_{k+1}$ then $d_l$ is unoccupied when $c_{k+1}$ tries to park. By assumption $d_{k+1}<d_l$, which forces $c_{k+1}$ to park at spot $d_l$ or earlier. If $c_{k+1}$ parks in spot $d_l$ then when cars $c_{k+2},\dots,c_n$ go to park according to $P_2$ they find spots $d_1,\dots, d_{k+1}$ occupied and park in spots $d_{k+2},\dots d_n$ respectively. Thus, we may assume that $c_{k+1}$ parks  between spots $d_{k+1}$ and $d_l$. Thus, the collection of spots $X$ occupied by cars $c_1,...,c_{k+1}$ differs as a set from  $X'=\{d_1,\dots, d_{k+1}\}$ by one element. Say $X\setminus X'=d'$ and $X'\setminus X=d''$ so that $d'<d''$. Then we can arrange the these spots into a $(k+1)$-filled parking preference satisfying our inductive hypothesis so that it is a $(k+1)$-filled $k$-Naples parking function. But that means that cars $c_{k+2},\dots c_n$ are able to park based on how cars $c_1,...,c_{k+1}$ have filled the lot according to $P_2$ so that $P_2$ is a $k$-filled Naples parking function.

If $p_l\neq d_{k+1}$ then it must be the case that $p_l=d_j$ for some $j>k+1$. 
Then when $c_{k+1}$ goes to park it can either park in $d_l$ or some earlier spot.
If it parks in $d_l$ we fall into the same situation as above of having a $(k+1)$-filled $k$-Naples parking function. 
If it parks in some earlier spot, this spot would have also been available to $c_{k+1}$ when it tried to park according to $P_1$ and thus is a contradiction.
\end{proof}

Lemma \ref{lem:t1} plays a key role in the following results about rearrangements of $k$-Naples parking functions.

\begin{thm}\label{thm:Rearrangements}
Given a parking preference, all of its rearrangements are $k$-Naples parking functions if and only if its ascending rearrangement is a $k$-Naples parking function.
\end{thm}

\begin{proof}
Note that if all rearrangements of a parking preference are $k$-Naples then this includes the fact that the ascending rearrangement is $k$-Naples.
To prove that all rearrangements of an ascending $k$-Naples parking function are $k$-Naples it suffices to show that if we have a $k$-Naples parking function $P_1=(a_1, a_2,\ldots,a_n)$, with $a_i<a_{i+1}$ for some $i,$ then the preference $P_2=(b_1,b_2,\ldots,b_n)$ where $b_j=a_j$ for every $j\notin\{i,i+1\},$ $b_i=a_{i+1},$ and $b_{i+1}=a_i,$ is also a $k$-Naples parking function.

Let car $c_j$ park in spot $d_j$ in accordance with parking preference $P_1$.
We see that both parking preferences $P_1$ and $P_2$ result in the first $i-1$ cars park identically, then $p_j=d_j$. 
If $c_i$ in $P_2$ parks in $d_{i+1}$, then $c_i$ in $P_1$ must not pass by $d_{i+1}$ before parking or it would park there.
The two cars take up the same spaces in $P_2$ and the rest of the parking proceeds as in $P_1$. 
So, we may assume $c_i$ in $P_2$ does not park in $d_{i+1}$.
But then it must be parking in a spot not open to $c_{i+1}$ in $P_1$, namely $d_i$. 
Then when $c_{i+1}$ goes to park in $P_2$, it drives past $d_i$ which is now full. 

If $d_{i+1}<d_i$, we see car $c_{i+1}$ in $P_1$ backed up all the way to $d_{i+1}$ and since $b_{i+1}<b_i$, $c_{i+1}$ in $P_2$ backs up to $d_{i+1}$ as well.
Otherwise, we have $d_{i+1}>d_i$ and $c_{i+1}$ in $P_2$ clearly parks at or before $d_{i+1}$. 
So, we know that after the $i+1$st car in $P_2$ parks, all the spots the first $i$ cars park in for $P_1$ are full, and a car is either parked in $d_{i+1}$ or in a spot that would be open in $P_1$ that is before $d_{i+1}$. 
Since this is the situation in the previous lemma, we see $P_2$ is a $k$-Naples parking function as desired. 
\end{proof}

\begin{exmp}
We saw above that $(6,6,5,5,3,1)$ is a $2$-Naples parking function, but its rearrangement $(3,5,1,6,6,5)$ is not. 
Note that in the ascending rearrangement $(1,3,5,5,6,6)$, no car can park in the second spot, so it is not a $2$-Naples parking function. 
However, we know $(1,3,3,5,6,6)$ is an ascending $2$-Naples parking function and an exhaustive search shows that all of the rearrangements are also $2$-Naples parking functions. 
\end{exmp}

\begin{rem} 
On a closer look, this proof actually defines a hierarchy that is followed when deciding when rearrangements of a parking preference are $k$-Naples.
If two rearrangements differ by just one switch, where the switched car that comes after in the first rearrangement has a higher preference than the one coming before, then it is intrinsically harder for the first one to be $k$-Naples than the second one. 
This is because according to the proof, the first being $k$-Naples implies the second also is, but the converse is not true. 
\end{rem}

\section{Dyck Paths}\label{sec:DyckPaths}
One family of Catalan objects that is in bijection with descending parking functions are Dyck paths. 
In \cite{naples}, a generalization of this result is presented, that gives a bijection between descending $k$-Naples parking functions and a generalization of Dyck paths called $k$-Dyck paths.  
In related work by Colmenarejo et al \cite{kNPF-AIMUP}, the authors counted k-Naples parking functions through permutations and they also defined the k-Naples area statistic.
In this section, we  explore when a $k$-Dyck path corresponds to an ascending $k$-Naples Parking function, giving a way of finding all $k$-Naples parking functions with the property that each of its rearrangements remains a $k$-Naples parking function. 
We then embed $k$-Dyck paths into a subset of Dyck paths to help us find other bijections with ascending and descending $k$-Naples parking functions.

\begin{defn}\label{def:DyckPath}
    A \textit{Dyck path} of length $n$ is a lattice path of Up $(1,1)$ and Down $(-1,-1)$ steps from $(0,0)$ to $(n,n)$ that never reaches below the line $y=0$. A \textit{$k$-Dyck path} is a similarly defined path that never reaches below the line $y=-k$ and ends with a Down step. 
    Any such path can be represented by a sequence of $U$'s and $D$'s corresponding to its steps. 
    The \textit{length} of a Dyck path or $k$-Dyck path is defined as the number of Up steps it has.
\end{defn}

The following result from \cite{naples} connects $k$-Dyck paths to descending $k$-Naples parking functions.

\begin{prop}[Theorem 1.3, \cite{naples}]\label{prop:DescendingKDyckBijection}
The set of descending $k$-Naples parking functions of length $n$ are in bijective correspondence to $k$-Dyck paths of length $n$. 
\end{prop}

\begin{rem}\label{rem:ParkingFunctionToDyck} From \cite{naples} we have the following correspondence between $k$-Dyck paths and increasing parking preferences.
A $k$-Dyck path $P$ of length $n$ uniquely corresponds to the parking preference $\alpha=(a_1,\dots, a_n)$, where $a_i$ is $1$ plus the number of Down steps coming before the $i$th Up step.
Note that, $\alpha$ is an ascending parking preference. In \cite{naples} the descending rearrangement of $\alpha$ was shown to be $k$-Naples, and it is straightforward to reverse this process to go from decreasing $k$-Naples parking functions of length $n$ to $k$-Dyck paths of length $n$.
\end{rem}

Next, we use this correspondence between ascending parking preferences of length $n$ and $k$-Dyck paths of length $n$ to classify which $k$-Dyck paths correspond to ascending $k$-Naples parking functions.

\begin{figure}[h]
\centering
 \begin{tikzpicture}[scale=0.75]
\draw[red, thick] (0,0) to (12,0);
\draw[thick] (0,0) to (1,1) to (2,0) to (3,-1) to (4,0) to (5,1) to (6,0) to (7,-1) to (8,0) to (9,-1) to (10,0) to (11,1) to (12,0);

\fill   (0,0) circle (2pt)
        (1,1) circle (2pt)
        (2,0) circle (2pt)
        (3,-1) circle (2pt)
        (4,0) circle (2pt)
        (5,1) circle (2pt)
        (6,0) circle (2pt)
        (7,-1) circle (2pt)
        (8,0) circle (2pt)
        (9,-1) circle (2pt)
        (10,0) circle (2pt)
        (11,1) circle (2pt)
        (12,0) circle (2pt)
        ;
\end{tikzpicture}

\caption{The $k$-Dyck path corresponding to $(1,3,3,5,6,6)$.} \label{fig:PPToKDyck}
\end{figure}
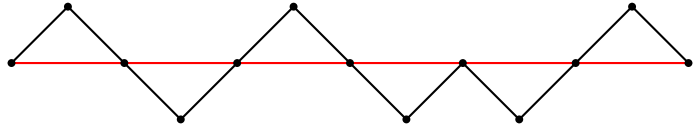

\begin{exmp}
To go from the 2-Dyck path in Figure \ref{fig:PPToKDyck} to an ascending parking preference, we see the first Up step has no previous Down steps making the first preference $1$.
The second and third Up step then correspond to a preference of $3$.
Continuing in this manner yields the parking preference $(1,3,3,5,6,6)$. 
\end{exmp}

\begin{thm}\label{thm:AscendingCharacterization}
    A $k$-Dyck path corresponds to an ascending $k$-Naples parking function if and only if every Down step that puts the path below the line $y=0$ crosses back above $y=0$ within $2k$ steps. 
\end{thm}

\begin{proof}
 Assume for the sake of contradiction that we have a $k$-Dyck path $P$ that corresponds to an ascending $k$-Naples parking function, but that at some point $P$ crosses below the line $y=0$ and does not cross back above this line within $2k$ steps. 
 Let step $2i+1$ be the first step where the path goes below $y=0$, but does not cross back above $y=0$ within $2k$ steps. 
 Now, let us look at car $c_{i+j}$ with $1\leq j\leq k+1.$ 
 We see $c_{i+j}$ must have preference larger than $i+j$ since the path is below the horizontal and there are more Down than Up steps during this section. 
 For the car to move back to spot $i+1$, all spots between $i+1$ and the parking preference, including the preference, must be filled.
 We see these are spots $i+2, i+3, \ldots, i+j+1$ of which there are $j$. 
 However, they could only be filled by cars $c_{i+1}, c_{1+2}, \ldots c_{i+j-1}$ of which there are $j-1$. 
 So, one of these spots is open and $c_{i+j}$ cannot fill spot $i+1$. We see cars $c_{i+k+2}$ and later must have parking preference at or larger than or equal to $i+k+2$ since the path does not go below the diagonal until at least step $2(i+k)+2$ by assumption.  So no car fills spot $i+1$ showing that the path does not correspond to a $k$-Naples parking function. 

Next, we show that if a $k$-Dyck path always crosses back above the line $y=0$ within $2k$ steps of it crossing below $y=0$ then it corresponds to an ascending parking function.
We justify this by induction on $k$.
We know this is true for the usual parking functions, i.e. $0$-Naples parking functions, and we assume it is true for all values up to $k-1$.
Suppose we have a $k$-Naples parking function corresponding to a $k$-Dyck path that goes below the line $y=0$ on step $2i+1$.
We may assume that the $k$-Naples parking function leads to the first $i$ spots being filled by the first $i$ cars. 

By hypothesis, the path must go above the horizontal at or before step $2(i+k)+1$.
If the $k$-Dyck path always goes above $y=0$ before step $2(i+k)+1$ after going below $y=0$ on step $2i+1$ then it corresponds to an ascending $(k-1)$-Naples parking function since it cannot touch the line $y=-k$ given that within $2(k-1)$ steps of going under the horizontal it has crossed back above $y=0$.
So, we assume the path goes back above the horizontal for the first time on step $2(i+k)+1$.

Now, we look at cars $c_j$ with $i+1\leq j\leq i+k+1$.
These are all the cars with preferred parking spot $a_j$ corresponding to Up steps below the horizontal except for car $c_{i+k+1}$ whose preferred parking spot corresponds to the last Up step to back above the horizontal $y=0$.
We see that their parking preferences $a_j$ have the property $i+1\leq a_j\leq i+k$.
Since each car is able to move back $k$ spots and $j\leq i+k+1$, we see these cars fill spots at or before spot $i+k+1$.
But that implies the $k+1$ cars fill up the $k+1$ spots immediately after what was already filled. 
So, the first $i+k+1$ cars fill the first $i+k+1$ spots.
Now, if the path goes below the horizontal for the first time on step $2i'+1$, then the first $i'$ spots are clearly filled. 
This shows that a $k$-Dyck path corresponds to an ascending $k$-Naples parking function when each time the path goes below the horizontal, it has had more Up steps than Down steps at least once within $2k$ steps. 
\end{proof}

\begin{cor}\label{cor:Ascending1-Naples}
Every rearrangement of a parking preference is a $k$-Naples parking function if and only if whenever its corresponding $k$-Dyck path has a Down step which crosses the line $y=0,$ the following $2k$ steps have a point where there have been in total two more Up steps than Down steps so far into the path.
\end{cor}

\begin{proof}
 This follows directly from Theorems \ref{thm:AscendingCharacterization} and \ref{thm:Rearrangements}.
\end{proof}
\begin{exmp}
From Figure \ref{fig:PPToKDyck}, we see $(1,3,3,5,6,6)$ is not a $1$-Naples parking function even though the lattice path is a $1$-Dyck path. 
In fact, we can see that at step 7 it crosses below $y=0$, and then it takes four steps for the lattice path for the path to cross back above $y=0$. 
Thus, it is a $2$-Naples parking function. 
\end{exmp}

\begin{rem}
Note that, in particular, for the 1-Naples case, the path cannot be below the $y=0$ line for more than 3 steps at a time.
This means that we cannot have two consecutive valleys under the $y=0$ line. 
This special case is equivalent to Conjecture 5 in \cite{naples}.
\end{rem}

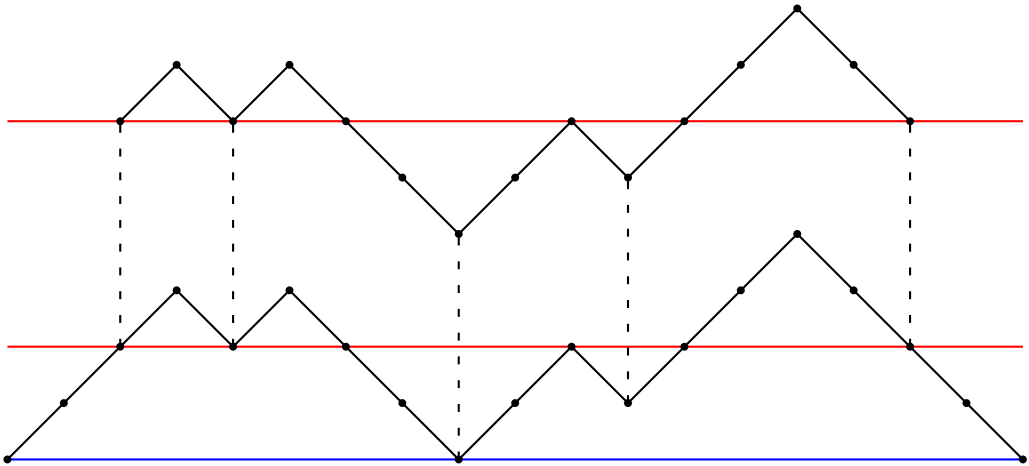
\begin{figure}[h]
\[	
\begin{tikzpicture}[scale=.75]
\newcommand{\U}{--++(1,1)} 
\newcommand{\D}{--++(1,-1)} 

\begin{scope}[shift=($(90:6)+(0:2)$)]

\draw[red, thick] (-2,0) to (16,0);

\draw[thick] (0,0) \U \D \U \D \D\D\U \U\D\U\U\U\D\D;

\fill 	
	(0,0) circle (2pt)
	(1,1) circle (2pt)
	(2,0) circle (2pt)
	(3,1) circle (2pt)
	(4,0) circle (2pt)
	(5,-1) circle (2pt)
	(6,-2) circle (2pt)
	(7,-1) circle (2pt)
	(8,0) circle (2pt)
	(9,-1)circle (2pt)
	(10,0)circle (2pt)
	(11,1) circle (2pt)
	(12,2) circle (2pt)
	(13,1) circle (2pt)
	(14,0)circle (2pt)
	;
\end{scope}

\draw[red, thick] (0,2) to (18,2);

\draw[blue, thick] (0,0) to (18,0);

\begin{scope}[shift=($(90:2)+(0:2)$)]
\draw[thick] (0,0) \U \D \U \D \D\D\U \U\D\U\U\U\D\D;

\fill 	
	(0,0) circle (2pt)
	(1,1) circle (2pt)
	(2,0) circle (2pt)
	(3,1) circle (2pt)
	(4,0) circle (2pt)
	(5,-1) circle (2pt)
	(6,-2) circle (2pt)
	(7,-1) circle (2pt)
	(8,0) circle (2pt)
	(9,-1)circle (2pt)
	(10,0)circle (2pt)
	(11,1) circle (2pt)
	(12,2) circle (2pt)
	(13,1) circle (2pt)
	(14,0)circle (2pt)
	;
\end{scope}

\draw[thick] (0,0) \U\U;

\draw[thick] (16,2) \D\D;
\fill 	
	(0,0) circle (2pt)
	(1,1) circle (2pt)

	(18,0) circle (2pt)
	(17,1) circle (2pt)
	;

\begin{scope}[shift=($(90:2)+(0:2)$), thick, loosely dashed]
\draw (0,0)--++(90:4);
\draw (2,0)--++(90:4);
\draw (6,-2)--++(90:4);
\draw (9,-1)--++(90:4);
\draw (14,0)--++(90:4);

\end{scope}
\end{tikzpicture}
\]
\caption{The $k$-Dyck path to Dyck path transformation.} 
\label{fig:KDyckToDyck}
\end{figure}

Next, we use Proposition \ref{prop:DescendingKDyckBijection} to view descending $k$-Naples parking functions of length $n$ as $k$-Dyck paths of the same length, and then embed these into usual Dyck paths of length $n+k$.
This allows us to obtain many similar bijections for both ascending and descending $k$-Naples parking functions to other subsets of Catalan objects.

\begin{prop}\label{prop:TranslatingAxis}
Descending $k$-Naples parking functions are in bijective correspondence to Dyck paths of length $n+k$ whose first $k$ steps are Up and last $k+1$ steps are Down.
\end{prop}

\begin{proof}
  This is a matter of using the bijection between descending $k$-Naples functions and $k$-Dyck paths and then embedding these $k$-Dyck paths into the usual Dyck paths. 
  Specifically, given a descending $k$-Naples parking function find the corresponding $k$-Dyck path.
  Then, shift this path $k$ units right and $k$ units up so that it starts at $(k,k)$ and concatenate this with the lattice path of all Up steps from $(0,0)$ to $(k,k)$ and the lattice path of all Down steps from $(2n+k,k)$ to $(2n+2k,0)$. In terms of Up steps and Down steps, this corresponds to appending $k$ Up steps to the start of the $k$-Dyck path and $k$ Down steps to the end of the $k$-Dyck path.
  
  Notice, that if a lattice path initially takes $k$ consecutive Up steps, then it is at $(k,k)$ after the first $k$ steps, and $y=0$ is $k$ units below this point.
  By concatenating this with a $k$-Dyck path that is shifted $k$ units right and $k$ units up to start at $(k,k)$, this new path does not venture below $y=0$. 
  Note that the $k$-Dyck path has $n$ Up steps and $n$ Down steps, with the last step always a Down step, so that along $k$ consecutive down steps from $(2n+k,k)$ to $(2n+2k,0)$ this concatenated path is a Dyck path of length $n+k$.
  
  Moreover, the reverse of this process, namely removing the first $k$ Up steps and the last $k$ Down steps of Dyck path whose first $k$ steps are Up and last $k+1$ steps are Down and then shifting the resulting lattice path left $k$ units and down $k$ units so that it starts at $(0,0)$ and ends at $(2n,0)$, results in a $k$-Naples parking function.
\end{proof}

\begin{rem}
An example of this transformation for our running example $(6,6,6,5,5,2,1)$ can be seen in Figure \ref{fig:KDyckToDyck}.
Notice that for a $k$-Dyck path, the corresponding Dyck path represents a descending parking function of length $n+k$ that ends in at least $k$ cars with preference $1$. This leads to the following result.
\end{rem}

\begin{cor}
Descending $k$-Naples parking functions of length $n$ are in bijective correspondence to descending parking functions of length $n+k$ which end with at least $k$ cars with preference $1$. 
\end{cor}

\begin{rem}
Similar to the transformation in \ref{prop:TranslatingAxis}, we see that Dyck paths of length $n+k$ that do not return to the line $y=0$ until the last step are in bijection to Dyck paths of length $n+k-1$ by removing the Up step and last Down step.
This motivates the next result. 
\end{rem}

\begin{defn}\label{defn:Strictlyk-Naples}
A parking preference is \textit{strictly} $k$-Naples if it is $k$-Naples but not $(k-1)$-Naples.
\end{defn}

\begin{prop}\label{prop:DescendingtoDyck}
The descending $k$-Naples parking functions that are not descending $(k-1)$-Naples parking functions are in bijective correspondence to Dyck paths of length $n+k$ whose first $k$ steps are Up and last $k+1$ steps are Down and which return to the line $y=0$ sometime before the last step.
\end{prop}

\begin{proof}
 To see this, use the same translation between a $k$-Dyck path and Dyck path as before.
 The Dyck paths that do not return to the horizontal correspond to the $k$-Dyck paths that do not reach the line $y=-k$. 
 If the descending path reaches at most the line $y=-k+1$, it corresponds to a descending $(k-1)$-Naples parking function. 
 This shows the Dyck paths which do return to the $y=0$ before the last step, and thus correspond to $k$-Dyck paths which reach the line $y=-k$, are in bijection with descending $k$-Naples parking functions. 
\end{proof}

Finally, we may also use this embedding of $k$-Dyck paths into Dyck paths to see which Dyck paths are in correspondence to ascending $k$-Naples parking functions. The following corollary follows directly from Theorem \ref{thm:AscendingCharacterization}.

\begin{cor}\label{cor:AscendingtoDyck}
    Ascending $k$-Naples parking functions are in bijective correspondence to Dyck paths with length $n+k$ whose first $k$ steps are Up, last $k+1$ steps are Down, and before the last $k+1$ steps, whenever a Down step puts the path below the line $y=k$, the following $2k$ steps have a point with two more Up then Down steps.
\end{cor}

\section{Binary Trees}\label{sec:BinaryTrees}

In the previous section, we found a bijection between ascending and descending $k$-Naples parking functions of length $n$ and subsets of Dyck paths of length $n+k$.
Since there is already a well-known bijection between Dyck paths and full binary trees (see \cite{Stanley}, for example) we find a bijection between ascending or descending $k$-Naples parking functions and a certain subset of full binary trees.
By trimming the leaves of the full binary tree, we can then obtain a bijection to a subset of binary trees. 

We begin this section by recalling the following definitions.

\begin{defn}\label{defn:RootedTrees}
A \textit{tree} is an undirected, connected graph with no cycles.
We say that a tree is \textit{rooted} if there is one vertex distinguished as the \textit{root}.
In a rooted tree an \textit{ancestor} of a vertex $v$ is a vertex that lies on the path from $v$ to the root, and we say that an ancestor of a vertex $v$ is the \textit{parent} of $v$ if it is also adjacent to $v$.
Analogously, a \textit{descendant} of a vertex $v$ is any vertex that has $v$ as its ancestor, and a \textit{child} of a vertex $v$ is any vertex that has $v$ as its parent. 
A \textit{leaf} is a vertex with no children. 
A \textit{binary} tree is a rooted tree in which every vertex has at most 2 children. 
If every vertex has either 0 or 2 children we say that the binary tree is \textit{full}. 
\end{defn}

\begin{defn}\label{def:TreesToDyck}
The bijection we use between full binary trees with $2n+1$ nodes and Dyck paths of length $n$ is defined recursively for a given tree, $T$ by $B(T) = UB(T_1)DB(T_2)$. 
Here $U$ and $D$ represent Up and Down steps respectively, $T_1$ is the subtree of $T$ whose root is the left child of $T$ and consists of all descendants of this left child, and $T_2$ is the subtree of the right child. 
If $T'$ consists of a single vertex then $B(T')$ is just the empty word of Up and Downs. 
For further reference see \cite{Stanley}.
\end{defn}

\begin{rem}
From Proposition \ref{prop:TranslatingAxis} and Corollary \ref{cor:AscendingtoDyck} we can see that Dyck paths corresponding to ascending or descending $k$-Naples parking functions must start with $k$ Up steps.
Thus, the corresponding full binary trees must have a path from the root to $k$ consecutive left children.
Describing the remaining structure of the trees coming from ascending or descending $k$-Naples parking functions are one of the focuses of this section.
\end{rem}

\begin{rem}\label{rem:fullToTrimmed}
It is well-known that there is a bijection between full binary trees with $2n+1$ vertices and binary trees with $n$ vertices given by removing the leaves and the edges to which they are incident in the full binary tree. 
To obtain a full binary tree on $2n+1$ vertices from a binary tree on $n$ vertices we add leaves to each node without two children until each of the original vertices has two children. 
Figure \ref{fig:DyckFullBinaryAndBinary} illustrates both of these bijections. 
From left to right we have the Dyck path corresponding to the strictly 2-Naples parking function $(6,6,6,5,5,2,1)$, then the full binary tree on $2n+1$ vertices with the leaves highlighted in red, and lastly the binary tree on $n$ vertices obtained by pruning the red leaves.
\end{rem}

\begin{figure}[h]
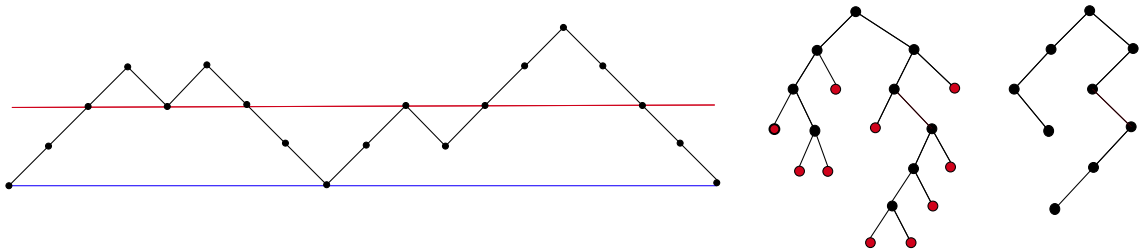

\centering



\caption{Dyck path to full binary tree to binary tree transformations.} \label{fig:DyckFullBinaryAndBinary}
\end{figure}


Before we determine which binary trees with $n+k$ nodes corresponds to ascending $k$-Naples parking functions we need the following definitions. 

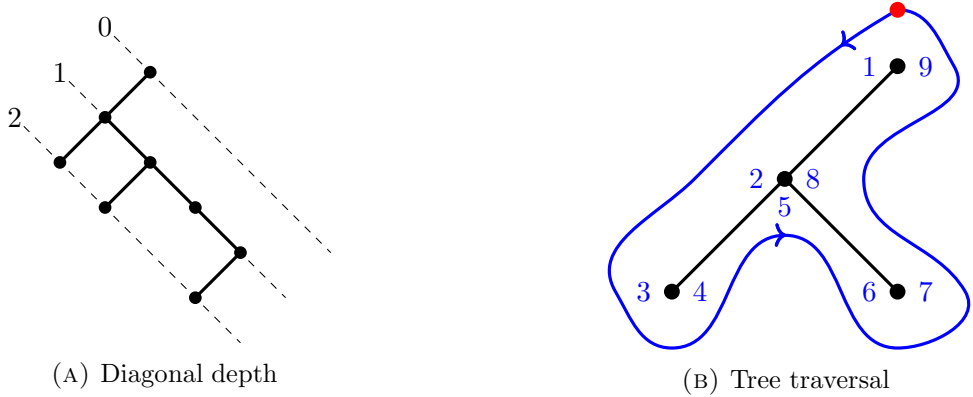
\begin{figure}[h]
\centering
\begin{subfigure}{.5\textwidth}
  \centering
  \begin{tikzpicture}[scale=0.6]
\draw[ dashed] 
	(-.8,.8) to (4,-4)
	(-1.8,-.2) to (3,-5)
	(-2.8,-1.2) to (2,-6)
	;

\draw[very thick] 
	(0,0) to (-2,-2)
	(-1,-1) to (2,-4)
	(0,-2) to (-1,-3)
	(2,-4) to (1,-5)
	;
	
\fill 
	(0,0) circle (4pt)
	(-1,-1) circle (4pt)
	(-2,-2) circle (4pt)
	(0,-2) circle (4pt)
	(-1,-3) circle (4pt)
	(1,-3) circle (4pt)
	(2,-4) circle (4pt)
	(1,-5)circle (4pt)
	;
	
\node at (-1,1) {$0$};
\node at (-2,0) {$1$};
\node at (-3,-1) {$2$};

\end{tikzpicture}
  \caption{Diagonal depth}
  \label{fig:diagdep}
\end{subfigure}%
\begin{subfigure}{.5\textwidth}
 \centering
 \begin{tikzpicture}[scale=1.5]
	\draw[very thick] 
		(0,0) to (-2,-2)
		(-1,-1) to (0,-2)
		;
		
	\draw[very thick, blue,->-=0.05,->-=0.5]
		(0,.5) to[out=-150, in=45] (-1.8,-1) to[out=-135, in=120] (-2.5,-2) to[out=-60, in=180] (-2,-2.5) to[out=0, in=180] (-1,-1.5) to[out=0, in=180] (0,-2.5) to[out=0, in=-60] (0.6, -2) to[out=120, in=-90] (-0.3,-1) to[out=90, in=-60] (0.5,0) to[out=120,in=0] cycle;	
		
	\fill[red] (0,.5) circle (2pt)	;
		
	\node[blue] at (-0.25,0) {$1$};
	\node[blue] at (-1.25,-1) {$2$};
	\node[blue] at (-2.25,-2) {$3$};
	\node[blue] at (-1.75,-2) {$4$};
	\node[blue] at (-1,-1.25) {$5$};
	\node[blue] at (-0.25,-2) {$6$};
	\node[blue] at (0.25,-2) {$7$};
	\node[blue] at (-0.75,-1) {$8$};
	\node[blue] at (0.25,0) {$9$};

	\fill 
		(0,0) circle (2pt)
		(-1,-1) circle (2pt)
		(-2,-2) circle (2pt)
		(0,-2) circle (2pt)
		;
	
\end{tikzpicture}

  \caption{Tree traversal}
  \label{fig:treetrav}
\end{subfigure}
\caption{Depth and traversal in binary trees.}
\end{figure}

\begin{defn}\label{def:DiagonalDepth}
The \textit{diagonal depth} of a node is the number of left steps needed in a path from the root to the node without using any edge more than once.
\end{defn}

\begin{defn}\label{def:TraversingNodes}
The \textit{traversal} of a rooted binary tree $T$ is the unique counterclockwise path in $T$ starting and ending at the root, traversing each edge exactly twice. A vertex $v$ of $T$ is \textit{traversed} at step $i$ if it is $i$-th vertex in the traversal.
\end{defn}

\begin{rem}
A tree traversal is illustrated in blue in Figure \ref{fig:treetrav} with each vertex labeled for each step at which it is traversed.
We see that the root is always traversed for the first time at step $1$ and its left child is traversed for the first time at step $2$.
Observe that a leaf is traversed for the first time at step $i$, then it is traversed for the second time at step $i+1$.
In general, each vertex that does not have two descendants is traversed twice, and each vertex with two descendants is traversed three times.
\end{rem}

To see when a rooted binary tree corresponds to an ascending $k$-Naples parking function of length, we need to understand what height the associated Dyck path is at while each node is being traversed. 

\begin{lem}\label{lem:TraversingHeights}
Suppose $T'$ is a rooted binary tree on $n$ vertices corresponding to the full rooted binary tree $T$ on $2n+1$ vertices, and suppose $B(T)$ is the corresponding Dyck path of length $n$. 
Let $v$ be a vertex of $T'$. If $v$ is traversed at steps $i<j$ in the traversal of $T'$, then the height of $B(T)$ immediately before step $i$ is the same as immediately after step $j$. 
\end{lem}

\begin{proof}
 Suppose step $i$ of $B(T)$ is an Up step, and let $v$ be the vertex traversed at step $i$ of the traversal of $T'$. By assumption $v$ is not a leaf in $T$. 
 Then the step $i$ Up step of $B(T)$ is directly followed by the path \[B(T_{v_L})DB(T_{v_R}),\] where $v_L$, $v_R$ are the left and right children of $v$ respectively and $T_{v_{L}}$, $T_{v_{L}}$ are the rooted subtrees those children define.
 Note that $B(T_{v_L})$ and $B(T_{v_R})$ both correspond to Dyck paths and thus must have the same number of Up and Down steps.
 Moreover, the Down step directly after $B(T_{v_L})$ is the $k$-th step of the Dyck path corresponding to the next time that vertex $v$ is traversed. 
 Consequently, the height of $B(T)$ before step $i$ and after step $k$ are identical. 
 
 Recall, that directly following the $k$-th down step of $B(T)$ is $B(T_{v_R})$. If $T_{v_R}$ is not a leaf then we have \[B(T_{v_R})=UB(T_{v_{RL}})DB(T_{v_{RR}}),\] where $T_{v_{RL}}$ ($T_{v_{RL}})$ is the rooted subtree of the left (right) child of $v_R$. 
 By the same reasoning as above the down step following $UB(T_{v_{RL}})$ in the above Dyck path occurs at some step $l$ that occurs at the same height as step $k$ and corresponds to $v$ being traversed at step $l$. 
 \end{proof}

\begin{rem}
Unless otherwise specified we consider the bijective correspondence between rooted binary trees and Dyck paths--where the number of nodes of the tree is the same as the length of the Dyck path--for the remained of this paper.
\end{rem}

\begin{cor}\label{cor:TraversingHeightsChild}
If a node of a rooted binary tree is traversed for the first time at step $i$ and its right child is traversed for the first time at step $j$, then the height of the corresponding Dyck path before step $i$ is the same as before step $j$. 
\end{cor}

\begin{proof}
Observe that if a node being traversed for the second time at step $k$ has a right child, then that right child is traversed for the first time at the next step. Along with the previous lemma this completes our proof.
\end{proof}

We can now use these results to see what a return to the horizontal line $y=0$ in the Dyck path corresponds to in the binary tree.

\begin{cor}\label{cor:ReturntoDiagonal}
The Dyck path corresponding to a rooted binary tree returns to the horizontal at step $i$ if $i$ is the last step or a direct right descendant of the root is traversed for the first time at step $i+1$. 
\end{cor}

\begin{proof}
This is true for the first descendant of the root since if the root is traversed for the second time at step $i$, the right descendant is traversed for the first time at step $i+1$. 
The rest follows by induction on the right descendants. 
\end{proof}

Using this result, we find that the height of the corresponding location of a Dyck path is based on the diagonal depth of the node. 

\begin{lem}\label{lem:HTraversingHeight}
If a node has diagonal depth of $h$ and is traversed for the first time on step $i$, then the height of the Dyck path before step $i$ is also $h$. 
\end{lem}

\begin{proof}

We see this is true for $h=0$. 
Now assume it is true for all values before $m$ and that the node $v$ has diagonal depth $m$ and is traversed for the first time at step $i$. 
We see that if $v$ is the right child of a node, then the height of the path before step $i$ is the same as the height before the step the parent of $v$ is traversed for the first time. 
This allows us to look at the most recent ancestor of $v$ that was the left child. 
Since we are assuming $h>0$, this must exist.
Let this ancestor be node $w$. 
By induction, we see that the path before the parent of $w$ is first traversed is at height $h-1$. 
But this next step is to a left descendent, so must be Up. 
This implies immediately before $w$ is first traversed, the path height is $h$. 
It follows the same is true for $v$, thus proving the statement.
\end{proof}

\begin{rem}
Lemma~\ref{lem:HTraversingHeight} has many consequences. 
For one, we know when the final node of a rooted binary tree is traversed for the first time at step $i$, every remaining step of the corresponding length $2n$ binary tree is Down.
Notice that step $i$ is an Up step. So after step $i$, the path must have a height of at least $k+1$, requiring the path to have a height of at least $k$ before step $i$.
This implies the final node traversed must have a diagonal depth of $k$.
\end{rem}

Now, we look at the nodes with diagonal depths $k-1$. 
When these nodes are traversed for a second time, the corresponding path is taking a Down step to the line $y=k-1$. 
Recall that this is the same as going below the line $y=0$ in a $k$-Dyck path.
Whenever this happens, at some point over the next $2k$ steps, there must be $2$ more Up then Down steps. 
For the tree, this corresponds to one of the next $2k-1$ nodes being traversed, either for the first or second time, must have diagonal depth $k$. 
This completes the description of binary trees that correspond to ascending $k$-Naples parking functions. 

\begin{rem}
Notice that all observations in this section apply equally to the trees which correspond to descending $k$-Naples parking functions except that the corresponding Dyck paths have no restriction on the number of steps spent under the line $y=k$.
\end{rem}

We now give another bijection involving a subset of binary trees, this time from descending strictly $k$-Naples parking functions. 
For this, notice that descending strictly $k$-Naples parking functions are in bijection with Dyck paths of length $n+k$ that start with $k$ Up steps, return to the horizontal before they end, and have at least $k+1$ Up steps after the first return to the horizontal. 
We already know that Dyck paths that start with $k$ Up steps, end with $k+1$ Down steps, and return to the horizontal are in bijection to descending strictly $k$-Naples parking functions, so reflecting the Dyck path after the first return gives us the new result. 
This can be seen in Figure \ref{fig:StrictlyKBijections}. 

\begin{figure}[h]
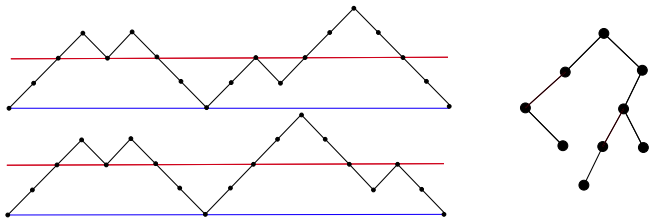

    \centering

\caption{Bijections for descending strictly $k$-Naples parking functions.} 
\label{fig:StrictlyKBijections}
\end{figure}

\begin{prop}\label{prop:SecondBinaryTree}
Descending strictly $k$-Naples parking functions are in bijection with binary trees that have $n+k$ nodes and satisfy the properties that the root has at least $k-1$ left children in a row, has a right child, and this right child has at least $k$ left children in a row.
\end{prop}

\begin{proof}
To see this, notice that the left subtree of the root remains unchanged. 
For the right subtree, this is equivalent to having a return to the horizontal, then using our earlier bijection we see the next $k+1$ steps are Up.
There is then no restriction on the ending. 
So, we see this is indeed in bijection to the Dyck paths just mentioned, and therefore to descending strictly $k$-Naples parking functions.
An example of this binary tree can be seen in Figure \ref{fig:StrictlyKBijections}. 
\end{proof}

\section{Enumeration of Monotonic $k$-Naples Parking Functions}\label{sec:Monotonic}

In the previous two sections we found bijections between either of the two types of monotonic $k$-Naples parking functions--descending or ascending--and subsets of both Dyck paths and binary trees. 
We now use those bijections to help us count these objects.
In particular, we give a recursive formula for the number of ascending $k$-Naples parking functions and a closed formula for their descending counterparts.
We also give results about the generating functions for the sequences corresponding to these objects. 
Throughout this section $C(x)$ is the generating function for the Catalan numbers, and $C_k$ is the $k$th Catalan number.

\begin{defn}\label{defn:Convolution}
Given sequences $(a_n)=a_0,a_1,a_2,\dots$ with generating function $A(x)=a_0+a_1x+a_2x^2+\cdots$ and $(b_n)=b_0,b_1,b_2,\dots$ with generating function $B(x)=b_0+b_1x+b_2x^2+\cdots$, the \textit{convolution} of $(a_n)$ and $(b_n)$ is defined to be $g_n=\sum_{i=0}^na_ib_{n-i},$ with generating function $G(x)=A(x)B(x).$ 
If $B(x)=A(x)$ we write $G(x)=A^2(x)$ and similarly extend to higher exponents. 
\end{defn}



\begin{thm}\label{thm:RecurrenceAscending}
Let $I_{n,k}$ denote the number of ascending $k$-Naples parking functions of length $n$, and let $U_{n,k}$ denote the number of ascending $k$-Naples parking functions of length $n$ which start with~$1$. 
For $n-1\geq k\geq1$ and $n\geq0$, we have 
\begin{align}
    I_{n,k} = I_{n,k-1}+C_k\sum^{n-k}_{i=0} (I_{i,k-1})(U_{n-k-i,k})\text{ and}\\ U_{n,k} = U_{n,k-1}+\sum^{n-k}_{i=0} (U_{i,k-1})(C_k)(U_{n-k-i,k}) \label{eq:5.2}.
\end{align}
\end{thm}

\begin{rem}
Note that $I_{n,0} = C_n$ and $U_{0,k} = 0$, otherwise $U_{n,0} = C_n$.
Further observe that the $k$-Naples parking functions which start with $1$ correspond to $k$-Dyck paths that start with an Up step.
In Theorem \ref{thm:RecurrenceAscending}, if $n\leq k$, then both summations are empty making them $0$.
This corresponds to there being no new $k$-Naples for a fixed length $n$ if $k$ is large enough.
\end{rem}

\begin{figure}[h]
    \centering
\begin{tikzpicture}[x=0.75pt,y=0.75pt,yscale=-1,xscale=1]

\draw [color={rgb, 255:red, 208; green, 2; blue, 27 }  ,draw opacity=1 ][fill={rgb, 255:red, 208; green, 2; blue, 27 }  ,fill opacity=1 ]   (370.56,172.25) -- (130.94,172.75) ;
\draw [color={rgb, 255:red, 139; green, 87; blue, 42 }  ,draw opacity=1 ][fill={rgb, 255:red, 0; green, 0; blue, 0 }  ,fill opacity=1 ]   (370.56,172.25) -- (351.09,152.25) ;
\draw [color={rgb, 255:red, 139; green, 87; blue, 42 }  ,draw opacity=1 ][fill={rgb, 255:red, 0; green, 0; blue, 0 }  ,fill opacity=1 ]   (351.09,152.25) -- (331.12,172.25) ;
\draw [color={rgb, 255:red, 139; green, 87; blue, 42 }  ,draw opacity=1 ][fill={rgb, 255:red, 0; green, 0; blue, 0 }  ,fill opacity=1 ]   (291.19,172.25) -- (311.23,151.75) ;
\draw [color={rgb, 255:red, 139; green, 87; blue, 42 }  ,draw opacity=1 ][fill={rgb, 255:red, 0; green, 0; blue, 0 }  ,fill opacity=1 ]   (311.23,151.75) -- (331.12,172.25) ;
\draw [color={rgb, 255:red, 155; green, 155; blue, 155 }  ,draw opacity=1 ][fill={rgb, 255:red, 0; green, 0; blue, 0 }  ,fill opacity=1 ]   (291.19,172.25) -- (271.22,192.25) ;
\draw [color={rgb, 255:red, 155; green, 155; blue, 155 }  ,draw opacity=1 ][fill={rgb, 255:red, 0; green, 0; blue, 0 }  ,fill opacity=1 ]   (210.97,171.75) -- (230.47,191.25) ;
\draw    (150.91,153.25) -- (170.87,132.25) ;
\draw    (170.87,132.25) -- (190.84,152.25) ;

\draw    (150.91,152.75) -- (130.94,172.75) ;
\draw  [fill={rgb, 255:red, 0; green, 0; blue, 0 }  ,fill opacity=1 ] (189.34,151.75) .. controls (189.34,150.92) and (190.01,150.25) .. (190.84,150.25) .. controls (191.67,150.25) and (192.34,150.92) .. (192.34,151.75) .. controls (192.34,152.58) and (191.67,153.25) .. (190.84,153.25) .. controls (190.01,153.25) and (189.34,152.58) .. (189.34,151.75) -- cycle ;
\draw  [fill={rgb, 255:red, 0; green, 0; blue, 0 }  ,fill opacity=1 ] (169.37,132.25) .. controls (169.37,131.42) and (170.05,130.75) .. (170.87,130.75) .. controls (171.7,130.75) and (172.37,131.42) .. (172.37,132.25) .. controls (172.37,133.08) and (171.7,133.75) .. (170.87,133.75) .. controls (170.05,133.75) and (169.37,133.08) .. (169.37,132.25) -- cycle ;
\draw  [fill={rgb, 255:red, 0; green, 0; blue, 0 }  ,fill opacity=1 ] (129.44,172.75) .. controls (129.44,171.92) and (130.11,171.25) .. (130.94,171.25) .. controls (131.77,171.25) and (132.44,171.92) .. (132.44,172.75) .. controls (132.44,173.58) and (131.77,174.25) .. (130.94,174.25) .. controls (130.11,174.25) and (129.44,173.58) .. (129.44,172.75) -- cycle ;
\draw  [fill={rgb, 255:red, 0; green, 0; blue, 0 }  ,fill opacity=1 ] (149.41,152.75) .. controls (149.41,151.92) and (150.08,151.25) .. (150.91,151.25) .. controls (151.73,151.25) and (152.41,151.92) .. (152.41,152.75) .. controls (152.41,153.58) and (151.73,154.25) .. (150.91,154.25) .. controls (150.08,154.25) and (149.41,153.58) .. (149.41,152.75) -- cycle ;
\draw  [fill={rgb, 255:red, 0; green, 0; blue, 0 }  ,fill opacity=1 ] (369.06,172.25) .. controls (369.06,171.42) and (369.73,170.75) .. (370.56,170.75) .. controls (371.39,170.75) and (372.06,171.42) .. (372.06,172.25) .. controls (372.06,173.08) and (371.39,173.75) .. (370.56,173.75) .. controls (369.73,173.75) and (369.06,173.08) .. (369.06,172.25) -- cycle ;
\draw  [color={rgb, 255:red, 0; green, 0; blue, 0 }  ,draw opacity=1 ][fill={rgb, 255:red, 0; green, 0; blue, 0 }  ,fill opacity=1 ] (349.59,152.25) .. controls (349.59,151.42) and (350.26,150.75) .. (351.09,150.75) .. controls (351.92,150.75) and (352.59,151.42) .. (352.59,152.25) .. controls (352.59,153.08) and (351.92,153.75) .. (351.09,153.75) .. controls (350.26,153.75) and (349.59,153.08) .. (349.59,152.25) -- cycle ;
\draw  [fill={rgb, 255:red, 0; green, 0; blue, 0 }  ,fill opacity=1 ] (329.62,172.25) .. controls (329.62,171.42) and (330.3,170.75) .. (331.12,170.75) .. controls (331.95,170.75) and (332.62,171.42) .. (332.62,172.25) .. controls (332.62,173.08) and (331.95,173.75) .. (331.12,173.75) .. controls (330.3,173.75) and (329.62,173.08) .. (329.62,172.25) -- cycle ;
\draw  [color={rgb, 255:red, 0; green, 0; blue, 0 }  ,draw opacity=1 ][fill={rgb, 255:red, 0; green, 0; blue, 0 }  ,fill opacity=1 ] (309.73,151.75) .. controls (309.73,150.92) and (310.4,150.25) .. (311.23,150.25) .. controls (312.06,150.25) and (312.73,150.92) .. (312.73,151.75) .. controls (312.73,152.58) and (312.06,153.25) .. (311.23,153.25) .. controls (310.4,153.25) and (309.73,152.58) .. (309.73,151.75) -- cycle ;
\draw  [fill={rgb, 255:red, 0; green, 0; blue, 0 }  ,fill opacity=1 ] (289.69,172.25) .. controls (289.69,171.42) and (290.36,170.75) .. (291.19,170.75) .. controls (292.01,170.75) and (292.69,171.42) .. (292.69,172.25) .. controls (292.69,173.08) and (292.01,173.75) .. (291.19,173.75) .. controls (290.36,173.75) and (289.69,173.08) .. (289.69,172.25) -- cycle ;
\draw  [fill={rgb, 255:red, 0; green, 0; blue, 0 }  ,fill opacity=1 ] (209.47,171.75) .. controls (209.47,170.92) and (210.14,170.25) .. (210.97,170.25) .. controls (211.8,170.25) and (212.47,170.92) .. (212.47,171.75) .. controls (212.47,172.58) and (211.8,173.25) .. (210.97,173.25) .. controls (210.14,173.25) and (209.47,172.58) .. (209.47,171.75) -- cycle ;
\draw [color={rgb, 255:red, 155; green, 155; blue, 155 }  ,draw opacity=1 ][fill={rgb, 255:red, 0; green, 0; blue, 0 }  ,fill opacity=1 ]   (230.47,191.29) -- (249.87,170.25) ;
\draw [color={rgb, 255:red, 155; green, 155; blue, 155 }  ,draw opacity=1 ][fill={rgb, 255:red, 0; green, 0; blue, 0 }  ,fill opacity=1 ]   (249.87,170.25) -- (271.22,192.25) ;
\draw  [fill={rgb, 255:red, 0; green, 0; blue, 0 }  ,fill opacity=1 ] (248.5,171.75) .. controls (248.5,172.58) and (249.17,173.25) .. (250,173.25) .. controls (250.83,173.25) and (251.5,172.58) .. (251.5,171.75) .. controls (251.5,170.92) and (250.83,170.25) .. (250,170.25) .. controls (249.17,170.25) and (248.5,170.92) .. (248.5,171.75) -- cycle ;
\draw    (191.47,152.25) -- (210.97,171.75) ;
\draw  [fill={rgb, 255:red, 0; green, 0; blue, 0 }  ,fill opacity=1 ] (269.72,192.25) .. controls (269.72,191.42) and (270.39,190.75) .. (271.22,190.75) .. controls (272.05,190.75) and (272.72,191.42) .. (272.72,192.25) .. controls (272.72,193.08) and (272.05,193.75) .. (271.22,193.75) .. controls (270.39,193.75) and (269.72,193.08) .. (269.72,192.25) -- cycle ;
\draw  [fill={rgb, 255:red, 0; green, 0; blue, 0 }  ,fill opacity=1 ] (228.97,191.29) .. controls (228.97,190.46) and (229.65,189.79) .. (230.47,189.79) .. controls (231.3,189.79) and (231.97,190.46) .. (231.97,191.29) .. controls (231.97,192.12) and (231.3,192.79) .. (230.47,192.79) .. controls (229.65,192.79) and (228.97,192.12) .. (228.97,191.29) -- cycle ;
\draw [color={rgb, 255:red, 0; green, 115; blue, 255 }  ,draw opacity=1 ]   (210.72,159) -- (211.22,187.5) ;
\draw [color={rgb, 255:red, 0; green, 115; blue, 255 }  ,draw opacity=1 ]   (291.22,157.88) -- (291.15,186.63) ;

\end{tikzpicture}

\caption{Breakdown for the recurrence for ascending $k$-Naples parking functions.} 
\label{fig:AscendingRecurrence}
\end{figure}

\begin{proof}

For $I_{n,k}$, we need to find the new ascending $k$-Naples parking functions which are not represented in $I_{n,k-1}$ and add the two together.
Recall that $I_{n,0} = C_n$ when $n>0$, giving the base for our recurrence.
From Theorem \ref{thm:AscendingCharacterization}, we know that an ascending $k$-Naples parking function of length $n$ that is not a $(k-1)$-Naples parking functions must have a corresponding Dyck path that is below the horizontal for exactly $2k$ steps.
Let there be $2i$ steps before the point it goes below the horizontal for $2k$ steps.
We see $i$ of these are Up steps since the last step must be to the horizontal. 
Also, notice that the last step is a Down step as otherwise the path would be below the horizontal for at least $2k+2$ steps.
So, the number of ways to get to this point is the number of $(k-1)$-Naples parking functions of length $i$, recalling that the last step of the $k$-Dyck paths corresponding to ascending Naples parking functions must be down. 
So, we see there are $I_{i,k-1}$ such paths. 
This argument corresponds to the first section in Figure \ref{fig:AscendingRecurrence}. 

Now, once the path has gone below the horizontal, it must stay there for $2k$ steps. 
At this point, it must return to the horizontal. 
Flipping this across the horizontal gives regular Dyck paths of length $k$, leading to $C_k$ possibilities. 
This argument corresponds to the second section in Figure \ref{fig:AscendingRecurrence}. 

Now, the path is at the horizontal after $2(i+k)$ steps.
It must then go up so as to not stay below the horizontal for too long.
But we see the final section of length $n-i-k$ starting with an Up step has $U_{n-i-k, k}$ options, which can be seen in the third section of Figure \ref{fig:AscendingRecurrence}. 

Summing over all $i$ gives $I_{n,k} = I_{n,k-1}+\sum^{n-k}_{i=0} (I_{i,k-1})(C_k)(U_{n-k-i,k})$.
Notice that the first term in the sum represents the paths to the horizontal, the second term the paths under the horizontal, and the third the ending paths which must start with an Up step. 
Also notice the $C_k$ is constant and can be pulled out, giving $I_{n,k-1}+C_k\sum^{n-k}_{i=0} (I_{i,k-1})(U_{n-k-i,k})$.

For $U_{n,k}$, we follow a similar argument. 
Again, recall that $U_{0,0} = 0$ and that otherwise $U_{n,0} = C_n$. 
This gives the base for our recurrence.
We now find the new paths that are $k$-Naples parking functions and not $(k-1)$-Naples parking functions. 
We see that eventually the corresponding $k$-Dyck path must be under the horizontal for $2k$ steps, then must have an Up step to above the horizontal, and finally it finishes the path. 
So, similarly as above, this gives $U_{n,k} = U_{n,k-1}+\sum^{n-k}_{i=0} (U_{i,k-1})(C_k)(U_{n-k-i,k})$ which can be rearranged to give
$$U_{n,k-1}+C_k\sum^{n-k}_{i=0} (U_{i,k-1})(U_{n-k-i,k}).$$ 
Also note that since $U_{0,k} = 0$, the indices can start at $i=1$. 

We see these sums only depend on smaller values of $n$ and $k$. 
Since we know the values of $I_{n,0}$ and $U_{n,0}$, we may find the value of $U_{n,k}$ independently of $I_{n,k}$ and vise versa.
So this gives a recurrence formula for the ascending $k$-Naples parking functions as well as the ascending $k$-Naples parking functions which start at $1$. 
\end{proof}

Given this recursive formula, we can see a connection between the $1$-Naples paths that start with an Up step the Fine Numbers, an integer sequence closely related to the Catalan Numbers. 
Many interpretations of the Fine Number sequence can be found in \cite{ShapiroFineNumber}.

\begin{thm}\label{thm:ClosedStartsUp}
For $n\geq 0$, we have $U_{n,1}=F_{n+1},$ where $F_{n+1}$ denotes the $(n+1)$th Fine Number (\textcolor{blue}{\href{https://oeis.org/A000957}{A000957}}). 
\end{thm}

\begin{proof}
Let $U_1(x)$ represent the ordinary generating function for $U_{n,1},$ $U_0(x)$ represent the one for $U_{n,0},$ $C(x)$ the one for the Catalan numbers, and $F(x)$ the one for the Fine numbers. 
It suffices to prove that $U_1(x)=\frac{F(x)-1}{x}$ (excluding $F_1$ and reindexing).
We know that $U_{n,0}=C_n$ for all $n>0,$ and $U_{0,0}=0=C_0-1;$ so we have $U_0(x)=C(x)-1.$ Now, for any given $n>0$ (and $k=1$), from \eqref{eq:5.2} we have
\begin{align*}
    U_{n,1}&=U_{n,0}+C_1\sum_{i=0}^{n-1}U_{i,0}U_{n-i-1,1}\\
    x^nU_{n,1}&=x^nU_{n,0}+x^nC_n\sum_{i=0}^{n-1}U_{i,0}U_{n-i-1,1}\\
    x^nU_{n,1}&=x^nU_{n,0}+x\sum_{i=0}^{n-1}x^iU_{i,0}x^{n-i-1}U_{n-i-1,1}.\\
\end{align*}
Adding these equations for all $n>0$ we get 

\[ \sum_{i=1}^{\infty}x^nU_{n,1}=\sum_{i=1}^{\infty}x^nU_{n,0}+x\sum_{i=1}^{\infty}x^{n-1}\sum_{i=0}^{n-1}U_{i,0}U_{n-i-1,1}.\]

Noticing that $\sum_{i=0}^{n-1}x^iU_{i,0}x^{n-i-1}U_{n-i-1,1}$ is the $(n-1)$th term in the convolution of $U_{n,0}$ and $U_{n,1},$ that is, the coefficient of $x^{n-1}$ in $U_1(x)U_0(x)$, we have 
\begin{align*}
    U_1(x)-U_{0,1}&=U_0(x)-U_{0,0}+xU_1(x)U_0(x)\\
    U_1(x)-0&=U_0(x)-0+xU_1(x)U_0(x)\\
    U_1(x)&=U_0(x)+xU_1(x)U_0(x)\\
    U_1(x)&=C(x)-1+xU_1(x)(C(x)-1)\\
    U_1(x)&=\frac{C(x)-1}{1+x-xC(x)}.
\end{align*}

But we know that $F(x)=\frac{1}{1-x^2C^2(x)},$ and that $xC^2(x)-C(x)+1=0,$ so \begin{align*}
    U_1(x)&=\frac{C(x)-1}{1+x-xC(x)}\\
    &=\frac{x(C(x)-1}{x(1+x-xC(x)}\\
    &=\frac{1-(1+x-xC(x))}{x(1+x-xC(x)}\\
    &=\frac{1}{x(1+x-xC(x)}-\frac{1}{x}\\
    &=\frac{1}{x(1-x(C(x)-1)}-\frac{1}{x}\\
    &=\frac{1}{x(1-x^2C^2(x))}-\frac{1}{x}\\
    &=\frac{F(x)}{x}-\frac{1}{x}=\frac{F(x)-1}{x}.\qedhere
\end{align*}
\end{proof}

\begin{thm}\label{thm:ClosedAscending}
For $n\geq 0,$ $I_{n,1}=CF_n,$ where $CF$ is the convolution of the Catalan numbers with the Fine numbers (\textcolor{blue}{\href{https://oeis.org/A000958}{A000958}}).
\end{thm}

\begin{proof}
Let $I_1(x)$ be the ordinary generating function for $I_{n,1}$ and $I_0(x)$ be the one for $I_{n,0}$.
Since $I_{0,1}=I_{0,0}$ and $I_0(x)=C(x)$ we get, by a very similar argument as Theorem \ref{thm:ClosedStartsUp},
\begin{align*}
    I_1(x)&=I_0(x)+xI_0(x)U_1(x)\\
    &=C(x)+xC(x)\frac{F(x)-1}{x}\\
    &=C(x)+C(x)F(x)-C(x)=C(x)F(x).
\end{align*} 
\end{proof}

\begin{thm}\label{thm:RecurrenceAscendingFunctional}
Let $U_k(x)$ represent the ordinary generating function for $U_{k,1},$  and define $U_{k-1}(x),$ $I_k(x),$ and $I_{k-1}(x)$ similarly, with $c_k$ being the $k$-th Catalan number.
Then
\begin{align}
    I_k(x)&=I_{k-1}(x)+c_kx^kI_{k-1}(x)U_k(x)\text{ and}\\
    U_k(x)&=U_{k-1}(x)+c_kx^kU_{k-1}(x)U_k(x).
 \end{align}
\end{thm}

\begin{proof}
The idea of the proof is identical to Theorem \ref{thm:ClosedAscending}, except for a given $k$ we are only able to use recurrence relations for $n\geq k,$ otherwise the convolution sum becomes meaningless.
This means that we are never adding the coefficients representing degrees smaller than $k$ for both $I_k(x)$ and $I_{k-1}(x)$ in (5.3) and conversely for $U$ in (5.4). 
But this is not an issue, since for $n<k$ any ascending parking preference is $k-1$-Naples (and thus also $k$-Naples), so the coefficients are the same and adding them on both sides of the equation do not change the result, so we can use the exact same reasoning as for Theorem \ref{thm:ClosedStartsUp}.
\end{proof}

For $k\geq 2,$ the generating functions become increasingly cumbersome to work with towards finding a closed formula.

We now use the bijective results between Dyck paths and Binary Trees to find a closed formula for the number of descending strictly $k$-Naples parking functions.
First, we give a closed formula for terms of the convolution of Catalan numbers.

\begin{lem}[Lemma 9, \cite{convolution}]\label{ClosedCOnvolutionFormula}
The $m$th term of the $r$th convolution of the Catalan numbers is given by $\frac{r}{2m-r}\binom{2m-r}{m}$.
\end{lem}

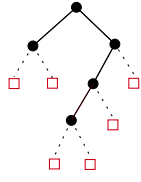
\begin{figure}[h]
    \centering
    
\begin{tikzpicture}[x=0.75pt,y=0.75pt,yscale=-1,xscale=1]

\draw  [fill={rgb, 255:red, 0; green, 0; blue, 0 }  ,fill opacity=1 ] (82.16,32.66) .. controls (82.16,31.29) and (83.27,30.18) .. (84.63,30.18) .. controls (86,30.18) and (87.11,31.29) .. (87.11,32.66) .. controls (87.11,34.03) and (86,35.14) .. (84.63,35.14) .. controls (83.27,35.14) and (82.16,34.03) .. (82.16,32.66) -- cycle ;
\draw [fill={rgb, 255:red, 0; green, 0; blue, 0 }  ,fill opacity=1 ]   (115,51.75) -- (125.98,31.75) ;
\draw  [fill={rgb, 255:red, 0; green, 0; blue, 0 }  ,fill opacity=1 ] (112.52,51.75) .. controls (112.52,50.38) and (113.63,49.27) .. (115,49.27) .. controls (116.37,49.27) and (117.48,50.38) .. (117.48,51.75) .. controls (117.48,53.12) and (116.37,54.23) .. (115,54.23) .. controls (113.63,54.23) and (112.52,53.12) .. (112.52,51.75) -- cycle ;
\draw [fill={rgb, 255:red, 0; green, 0; blue, 0 }  ,fill opacity=1 ]   (106.58,13.15) -- (84.63,32.66) ;
\draw  [fill={rgb, 255:red, 0; green, 0; blue, 0 }  ,fill opacity=1 ] (104.11,13.15) .. controls (104.11,11.78) and (105.22,10.67) .. (106.58,10.67) .. controls (107.95,10.67) and (109.06,11.78) .. (109.06,13.15) .. controls (109.06,14.52) and (107.95,15.63) .. (106.58,15.63) .. controls (105.22,15.63) and (104.11,14.52) .. (104.11,13.15) -- cycle ;
\draw [fill={rgb, 255:red, 0; green, 0; blue, 0 }  ,fill opacity=1 ]   (106.58,13.15) -- (125.98,31.75) ;
\draw  [fill={rgb, 255:red, 0; green, 0; blue, 0 }  ,fill opacity=1 ] (123.5,31.75) .. controls (123.5,30.38) and (124.61,29.27) .. (125.98,29.27) .. controls (127.34,29.27) and (128.45,30.38) .. (128.45,31.75) .. controls (128.45,33.12) and (127.34,34.23) .. (125.98,34.23) .. controls (124.61,34.23) and (123.5,33.12) .. (123.5,31.75) -- cycle ;
\draw  [fill={rgb, 255:red, 0; green, 0; blue, 0 }  ,fill opacity=1 ] (101.52,70.25) .. controls (101.52,68.88) and (102.63,67.77) .. (104,67.77) .. controls (105.37,67.77) and (106.48,68.88) .. (106.48,70.25) .. controls (106.48,71.62) and (105.37,72.73) .. (104,72.73) .. controls (102.63,72.73) and (101.52,71.62) .. (101.52,70.25) -- cycle ;
\draw [fill={rgb, 255:red, 208; green, 2; blue, 27 }  ,fill opacity=1 ]   (104,70.25) -- (115,51.75) ;
\draw  [color={rgb, 255:red, 208; green, 2; blue, 27 }  ,draw opacity=1 ] (92.98,89.75) -- (97.98,89.75) -- (97.98,94.75) -- (92.98,94.75) -- cycle ;
\draw  [color={rgb, 255:red, 208; green, 2; blue, 27 }  ,draw opacity=1 ] (122.5,70) -- (127.5,70) -- (127.5,75) -- (122.5,75) -- cycle ;
\draw  [color={rgb, 255:red, 208; green, 2; blue, 27 }  ,draw opacity=1 ] (72.63,49.16) -- (77.63,49.16) -- (77.63,54.16) -- (72.63,54.16) -- cycle ;
\draw  [color={rgb, 255:red, 208; green, 2; blue, 27 }  ,draw opacity=1 ] (111,90) -- (116,90) -- (116,95) -- (111,95) -- cycle ;
\draw  [color={rgb, 255:red, 208; green, 2; blue, 27 }  ,draw opacity=1 ] (92,49) -- (97,49) -- (97,54) -- (92,54) -- cycle ;
\draw  [color={rgb, 255:red, 208; green, 2; blue, 27 }  ,draw opacity=1 ] (133,49) -- (138,49) -- (138,54) -- (133,54) -- cycle ;
\draw  [dash pattern={on 0.84pt off 2.51pt}]  (104,70.25) -- (95.48,88.25) ;
\draw  [dash pattern={on 0.84pt off 2.51pt}]  (115,51.75) -- (125,68.75) ;
\draw  [dash pattern={on 0.84pt off 2.51pt}]  (135.5,46.75) -- (125.98,31.75) ;
\draw  [dash pattern={on 0.84pt off 2.51pt}]  (84.63,32.66) -- (75,48.25) ;
\draw  [dash pattern={on 0.84pt off 2.51pt}]  (84.63,32.66) -- (94.5,48.75) ;
\draw  [dash pattern={on 0.84pt off 2.51pt}]  (104,70.25) -- (113.5,88.25) ;

\end{tikzpicture}
    \caption{Places to put binary trees on the tree corresponding to descending strictly $k$-Naples parking functions.}
    \label{fig:PlacingBinaryTrees}
\end{figure}

\begin{thm}\label{thm:NewDescendingNaples}
The number of descending strictly $k$-Naples parking functions of length $n$ is $$\frac{k+1}{n}\binom{2n}{n+k+1}.$$
\end{thm}
\begin{proof}

First, notice that if $n\leq k$ then every parking preference of length $n$ is already a $(k-1)$-Naples parking function since each car can move to the beginning if need be, so we cannot have any strictly $k$-Naples parking functions when $n-k-1<0$.

Now, we use the bijection given in Proposition \ref{prop:SecondBinaryTree}. 
We know that the root has $k-1$ direct left descendants, at least $1$ right child, and that right child has at least $k$ direct left descendants.
This uses $1+(k-1)+1+k = 2k+1$ total nodes. 
So, there are $n-k-1$ nodes left.  
Each of the $k-1$ direct left descendants of the root have one place to put another (possibly empty) binary tree and the last node has two.
The right child of the root also has one place, while each of the $k$ direct left descendants of the right child has one. 
Again, the final node will have an extra. 
This gives $k+(k+2) = 2k+2$ total places to place (possibly empty) binary trees. 

Since the number of binary trees with $i$ nodes corresponds to the $i$th Catalan number \cite{Stanley}, we see this gives the number of descending strictly $k$-Naples parking functions as $$\sum_{i_1+i_2+\cdots+i_{2k+2}=n-k-1} C_{i_1}C_{i_2}\cdots C_{i_{2k+2}}.$$
this is the $(2k+2)$ convolution of the Catalan numbers.
Since the indices sum to $n-k-1$, the number of descending strictly $k$-Naples parking functions of length $n$ is the $n+k+1$ term of the $2k+2$nd convolution of the Catalan numbers.
In other words, this gives the corresponding generating function as $x^{k+1}C(x)^{2k+2}$ where we real $C(x)$ is the generating function for the Catalan numbers.
We see that plugging in $m = n-k-1$ and $r = 2k+2$ into the formula from $\ref{ClosedCOnvolutionFormula}$ gives $\frac{k+1}{n}\binom{2n}{n+k+1}$, as we wanted.
\end{proof}

\begin{lem}[Lemma 8, \cite{convolution}]
Let $1\leq q\leq p\leq 2q-1$. Then
$$\sum_{i\geq0}C_i\binom{p-1-2i}{q-1-i} = \binom{p}{q} .$$
\end{lem}

\begin{lem}\label{lem:lemma9}
Let $G(x) = \frac{C^2(x)}{1-xC^2(x)}$ be the generating functions for $\binom{2n+1}{n}$ (\textcolor{blue}{\href{https://oeis.org/A001700}{A001700}}). 
Then $G(x)(C(x)-1)^{k}$ is the generating functions for $\binom{2n+1}{n+k+1}.$
\end{lem}

\begin{proof}
For $k=0$, $G(C(x)-1)$ is the generating function for the convolution $\sum_{i=1}^n\binom{2(n-i)+1}{n-i+1}C_i$ where we use the closed formula for $G$, and the lower bound comes from the fact we are dealing with the Catalan numbers without the zeroth term. 
Then 
\begin{align*}
    \sum_{i=1}^n\binom{2(n-i)+1}{n-i+1}C_i&= \sum_{i=0}^n\binom{2(n-i)+1}{n-i+1}C_i - \binom{2n+1}{n+1}\\
    &= \sum_{i=0}^n\binom{(2n+2)-2i-1}{(n+2)-i-1}C_i - \binom{2n+1}{n+1}.
\end{align*}
Now, by Lemma \ref{lem:lemma9} with $p = 2n+2$ and $q=n+2$, we see that $$\binom{2n+2}{n+2} - \binom{2n+1}{n+1} = \binom{2n+2}{n} - \binom{2n+1}{n} = \binom{2n+1}{n-1} = \binom{2n+1}{n+2}$$ by using the recurrence relation for and the symmetry of binomial coefficients. 
This gives the formula for $k=0$. 
Now, assume it is true for all values less than $k$. This gives the convolution for $$G(x)(C(x)-1)^{k+1} = G(x)(C(x)-1)^{k}(C(x)-1)$$ as $\sum_{i=1}^n\binom{2(n-i)+1}{n-i+(k-1)+2}C_i$.
Following the same reasoning as before, we can see this is equivalent to $$\sum_{i=0}^n\binom{(2n+2)-i-1}{(n+k+2)-i-1}C_i - \binom{2n+1}{n+k+1}.$$
Again, by Lemma \ref{lem:lemma9} with $p = 2n+2$ and $q = n+k+2$, we see this gives $\binom{2n+2}{n+k+2} - \binom{2n+1}{n+k+1}$ which then equals $\binom{2n+1}{n+k+2}$, that is, $\binom{2n+1}{n+(k+1)+1}$.
\end{proof}

\begin{cor}\label{cor:AllDescendingNaples}
The total number of descending $k$-Naples parking functions of length $n$ is $\binom{2n-1}{n} - \binom{2n-1}{n+k+1}$.
\end{cor}
\begin{proof}
From Theorem \ref{thm:NewDescendingNaples}, we know the generating function for descending strictly $k$-Naples parking functions is $x^{k+1}C^{2k+2}(x) = (xC^2(x))^{k+1}$ where $C(x)$ is the generating function for the Catalan numbers.
Now sum over all values up to $k$. 
Let us call the generating function for this $D(x)$;  we see that
\begin{align*}
    D(x) &= \sum_{i=0}^k(xC^2(x))^{i+1}\\
    &= xC^2(x)\sum_{i=0}^k(xC^2(x))^{i}\\
    &= xC^2(x)(\frac{1-(xC^2(x))^{k+1}}{1-xC^2(x)}).
\end{align*}
Now, let $G(x) = \frac{C^2(x)}{1-xC^2(x)}$, which is the generating function for $\binom{2n+1}{n}$. Then we see that $D(x) = xG(x) - x^{k+2}G(x)C^{2k+2}(x)$.
Since we know a closed formula for $xG(x)$, we need to find one for $x^{k+2}G(x)C^{2k+2}(x) = xG(x)(C(x)-1)^{k+1}$.
We have the closed formula for the sequences for which $G$ and $G(x)(C(x)-1)^{k+1}$ are generating functions.
To find a closed formula for the sequence $D(x) = xG(x) - xG(x)(C(x)-1)^{k+1}$, we simply must combine these and offset by one to account for the extra factor of $x$. This gives $\binom{2n-1}{n} - \binom{2n-1}{n+k+1}$ as desired. 
\end{proof}

\section{Other Bijections}\label{sec:OtherBijections}

Now, that we have a bijection between descending strictly $k$-Naples parking functions and a subset of binary trees, we can use this to find other bijections.
Specifically, we demonstrate a bijection to subsets of $r$-in$-s$ dissections and $r$-rooted non-crossing set partitions. 
These generalize the bijection between descending parking functions of length $n$ with triangulations of $(n+2)$-gons and non-crossing set partitions of $[n]$.

\begin{defn}
An \textit{$r$-in$-s$ dissection is a dissection} of a $s$-gon into a $r$-gon and $(s-r)$ triangles.
\end{defn}

\begin{defn}
An \textit{$r$-rooted non-crossing set partition} is a non-crossing set partition where one of the parts, the root, has size $r$.
\end{defn}

The next theorems are nice consequences of the bijection defined in Proposition \ref{prop:SecondBinaryTree}.
As mentioned in the proof of Theorem \ref{thm:NewDescendingNaples}, choosing a strictly descending $k$-Naples parking function is equivalent to choosing an ordered set of $2k+2$ binary trees.
We know these binary trees with $i$ nodes are in bijection to triangulations of an $(i+2)$-gon and non-crossing partitions of the set $[i]$ \cite{Stanley}.

\begin{thm}\label{thm:t1}
Descending strictly $k$-Naples parking functions of length $n$ are in bijection with $(2k+2)$-in$-(n+k+1)$ dissections, up to rotation, but with a distinguished edge on the $(2k+2)$-gon.
\end{thm}

\begin{proof}
To go from the dissections to the binary tree, we start in the special edge and see the outer polygon triangulation it is attached to (it could be empty), and the tree corresponding to this triangulation is the subtree that is the left descendent of the left most node on the original tree. 
Going clockwise from there, each next edge in the $(2k+2)$-gon similarly defines a triangulation that corresponds to the next descendent subtree going left to right.

To go from the binary tree to the dissection, we identify what the $2k+2$ subtrees are (some may be empty), and convert them into triangulations as in \cite{Stanley}.

We then combine them, in order, making an $n$-gon out of the edges that correspond to the root of each subtree. The edge corresponding to the right child of the tree is then the special edge. 
This gives us an $(n+k+1)$-gon, since every polygon has $i+2$ edges, where $i$ is the number of nodes in the subtree, but one of them will be internal in the $(2k+2)$-gon, so in total we have $n-k-1+2k+2=n+k+1$ sides. 
We can easily see that none of the dissecting lines exit the polygon, so we can make it convex without creating any issues. 
\end{proof}

The process illustrated in the proof of Theorem \ref{thm:t1} is exemplified in Figure \ref{fig:treeanddissection} for the $2$-Naples parking function $(6,6,6,5,5,2,1)$.
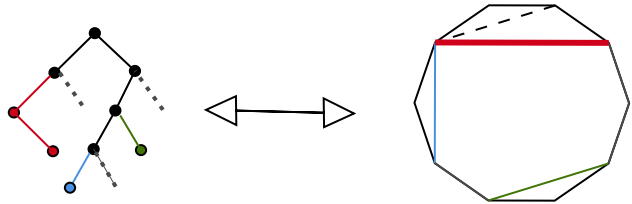
\begin{figure}[h]
    \centering
\tikzset{every picture/.style={line width=0.75pt}} 

\begin{tikzpicture}[x=0.75pt,y=0.75pt,yscale=-1,xscale=1]

\draw   (158.21,83.68) -- (173.31,76.79) -- (173.41,84.05) -- (217.66,85.13) -- (217.56,77.88) -- (232.85,85.5) -- (217.76,92.39) -- (217.66,85.13) -- (173.41,84.05) -- (173.5,91.31) -- cycle ;
\draw   (371.61,80.33) -- (361.19,110.75) -- (334.07,129.5) -- (300.61,129.42) -- (273.59,110.54) -- (263.34,80.07) -- (273.76,49.65) -- (300.88,30.9) -- (334.34,30.98) -- (361.36,49.87) -- cycle ;
\draw [color={rgb, 255:red, 74; green, 74; blue, 74 }  ,draw opacity=1 ]   (361.36,49.86) -- (371.62,80.33) ;
\draw [color={rgb, 255:red, 65; green, 117; blue, 5 }  ,draw opacity=1 ]   (361.19,110.75) -- (300.6,129.42) ;
\draw [color={rgb, 255:red, 74; green, 144; blue, 226 }  ,draw opacity=1 ]   (273.59,110.54) -- (273.76,49.65) ;
\draw  [dash pattern={on 4.5pt off 4.5pt}]  (334.35,30.98) -- (273.76,49.65) ;
\draw [color={rgb, 255:red, 74; green, 74; blue, 74 }  ,draw opacity=1 ]   (371.62,80.33) -- (361.19,110.75) ;
\draw [color={rgb, 255:red, 74; green, 74; blue, 74 }  ,draw opacity=1 ]   (273.59,110.54) -- (300.6,129.42) ;
\draw [color={rgb, 255:red, 208; green, 2; blue, 27 }  ,draw opacity=1 ][line width=2.25]    (273.76,49.65) -- (361.36,49.86) ;

\draw  [fill={rgb, 255:red, 0; green, 0; blue, 0 }  ,fill opacity=1 ] (79.23,64.96) .. controls (79.23,63.55) and (80.39,62.41) .. (81.82,62.41) .. controls (83.24,62.41) and (84.4,63.55) .. (84.4,64.96) .. controls (84.4,66.36) and (83.24,67.5) .. (81.82,67.5) .. controls (80.39,67.5) and (79.23,66.36) .. (79.23,64.96) -- cycle ;
\draw  [fill={rgb, 255:red, 65; green, 117; blue, 5 }  ,fill opacity=1 ] (122.84,103.99) .. controls (122.84,102.59) and (124,101.45) .. (125.43,101.45) .. controls (126.86,101.45) and (128.01,102.59) .. (128.01,103.99) .. controls (128.01,105.39) and (126.86,106.53) .. (125.43,106.53) .. controls (124,106.53) and (122.84,105.39) .. (122.84,103.99) -- cycle ;
\draw  [fill={rgb, 255:red, 0; green, 0; blue, 0 }  ,fill opacity=1 ] (98.95,103.48) .. controls (98.95,102.07) and (100.11,100.94) .. (101.53,100.94) .. controls (102.96,100.94) and (104.12,102.07) .. (104.12,103.48) .. controls (104.12,104.88) and (102.96,106.02) .. (101.53,106.02) .. controls (100.11,106.02) and (98.95,104.88) .. (98.95,103.48) -- cycle ;
\draw [fill={rgb, 255:red, 0; green, 0; blue, 0 }  ,fill opacity=1 ]   (112.49,84.01) -- (122.38,64.03) ;
\draw  [fill={rgb, 255:red, 0; green, 0; blue, 0 }  ,fill opacity=1 ] (109.91,84.01) .. controls (109.91,82.6) and (111.06,81.47) .. (112.49,81.47) .. controls (113.92,81.47) and (115.08,82.6) .. (115.08,84.01) .. controls (115.08,85.41) and (113.92,86.55) .. (112.49,86.55) .. controls (111.06,86.55) and (109.91,85.41) .. (109.91,84.01) -- cycle ;
\draw [fill={rgb, 255:red, 0; green, 0; blue, 0 }  ,fill opacity=1 ]   (102.14,44.96) -- (81.82,64.96) ;
\draw  [fill={rgb, 255:red, 0; green, 0; blue, 0 }  ,fill opacity=1 ] (99.56,44.96) .. controls (99.56,43.56) and (100.72,42.42) .. (102.14,42.42) .. controls (103.57,42.42) and (104.73,43.56) .. (104.73,44.96) .. controls (104.73,46.37) and (103.57,47.5) .. (102.14,47.5) .. controls (100.72,47.5) and (99.56,46.37) .. (99.56,44.96) -- cycle ;
\draw [color={rgb, 255:red, 65; green, 117; blue, 5 }  ,draw opacity=1 ][fill={rgb, 255:red, 65; green, 117; blue, 5 }  ,fill opacity=1 ]   (114.99,86.55) -- (125.43,103.99) ;
\draw [fill={rgb, 255:red, 0; green, 0; blue, 0 }  ,fill opacity=1 ]   (102.14,44.96) -- (122.38,64.03) ;
\draw [fill={rgb, 255:red, 208; green, 2; blue, 27 }  ,fill opacity=1 ]   (101.53,103.48) -- (112.49,84.01) ;
\draw  [fill={rgb, 255:red, 0; green, 0; blue, 0 }  ,fill opacity=1 ] (119.8,64.03) .. controls (119.8,62.62) and (120.95,61.48) .. (122.38,61.48) .. controls (123.81,61.48) and (124.96,62.62) .. (124.96,64.03) .. controls (124.96,65.43) and (123.81,66.57) .. (122.38,66.57) .. controls (120.95,66.57) and (119.8,65.43) .. (119.8,64.03) -- cycle ;
\draw  [fill={rgb, 255:red, 208; green, 2; blue, 27 }  ,fill opacity=1 ] (58.74,85) .. controls (58.74,83.59) and (59.9,82.45) .. (61.33,82.45) .. controls (62.76,82.45) and (63.91,83.59) .. (63.91,85) .. controls (63.91,86.4) and (62.76,87.54) .. (61.33,87.54) .. controls (59.9,87.54) and (58.74,86.4) .. (58.74,85) -- cycle ;
\draw  [fill={rgb, 255:red, 208; green, 2; blue, 27 }  ,fill opacity=1 ] (78.46,104.56) .. controls (78.46,103.16) and (79.62,102.02) .. (81.04,102.02) .. controls (82.47,102.02) and (83.63,103.16) .. (83.63,104.56) .. controls (83.63,105.96) and (82.47,107.1) .. (81.04,107.1) .. controls (79.62,107.1) and (78.46,105.96) .. (78.46,104.56) -- cycle ;
\draw [color={rgb, 255:red, 208; green, 2; blue, 27 }  ,draw opacity=1 ][fill={rgb, 255:red, 0; green, 0; blue, 0 }  ,fill opacity=1 ]   (61.33,85) -- (81.04,104.56) ;
\draw [color={rgb, 255:red, 208; green, 2; blue, 27 }  ,draw opacity=1 ][fill={rgb, 255:red, 208; green, 2; blue, 27 }  ,fill opacity=1 ]   (61.33,85) -- (79.73,66.46) ;
\draw [color={rgb, 255:red, 74; green, 144; blue, 226 }  ,draw opacity=1 ]   (99.53,105.52) -- (89.62,122.96) ;
\draw  [fill={rgb, 255:red, 74; green, 144; blue, 226 }  ,fill opacity=1 ] (87.03,122.96) .. controls (87.03,121.55) and (88.19,120.42) .. (89.62,120.42) .. controls (91.04,120.42) and (92.2,121.55) .. (92.2,122.96) .. controls (92.2,124.36) and (91.04,125.5) .. (89.62,125.5) .. controls (88.19,125.5) and (87.03,124.36) .. (87.03,122.96) -- cycle ;
\draw [color={rgb, 255:red, 74; green, 74; blue, 74 }  ,draw opacity=1 ][line width=1.5]  [dash pattern={on 1.69pt off 2.76pt}]  (84.4,64.96) -- (96.3,82.04) ;
\draw [color={rgb, 255:red, 74; green, 74; blue, 74 }  ,draw opacity=1 ][line width=1.5]  [dash pattern={on 1.69pt off 2.76pt}]  (136.5,83.65) -- (122.38,64.03) ;
\draw [color={rgb, 255:red, 74; green, 74; blue, 74 }  ,draw opacity=1 ][fill={rgb, 255:red, 74; green, 74; blue, 74 }  ,fill opacity=1 ][line width=1.5]  [dash pattern={on 1.69pt off 2.76pt}]  (112.67,122.38) -- (101.53,103.48) ;

\end{tikzpicture}

\caption{Trees and polygon dissections.}
    \label{fig:treeanddissection}
\end{figure}

\begin{thm}\label{thm:t2}
Descending strictly $k$-Naples parking functions with $n$ cars are in bijection with $(2k+2)$-rooted non-crossing partitions of $[n+k+1],$ where 1 is in the root.
\end{thm}

\begin{proof}
The proof for this theorem is very similar to that of Theorem \ref{thm:t1}. 
The root partitions the remaining $n-k-1$ elements into $2k+2$ subsets, where elements in each subset form a chain of consecutive elements (otherwise this overall partition would not be non-crossing).
Each of these subsets can now have a non-crossing partition of its elements that corresponds to a binary tree. 
The one starting right after the special element (could be empty) corresponds to the leftmost subtree, and clock-wise we have the ones corresponding to the remaining subtrees left to right. 

To go the other way we choose the root by selecting the element 1 and the next one will be however many nodes the first subtree has afterwards, plus 1, and so on. 
Each subtree now corresponds to a non-crossing partition of the subset it defined by its number of nodes.
\end{proof}
The process illustrated in the proof of Theorem \ref{thm:t2} is exemplified in Figure \ref{fig:treesandpartitions} for the $2$-Naples parking function $(6,6,6,5,5,2,1)$.
\begin{figure}[h]
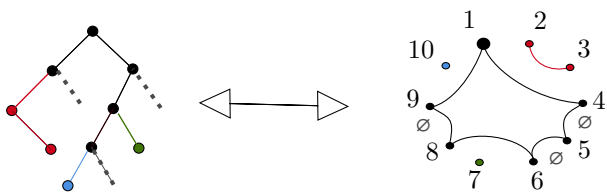

    \centering


    \caption{Trees and non-crossing partitions.}
    \label{fig:treesandpartitions}
\end{figure}

With these two final results at hand, in Figure \ref{fig:bijectionmap}, we illustrate the bijections we have established between $k$-Naples parking functions and (generalized) Catalan objects. 
It is no surprise that this figure gives a parallel with the objects in bijection to the (traditional) parking function example, as we illustrated in Figure \ref{fig:CatalanObjects}.

\begin{figure}[h]
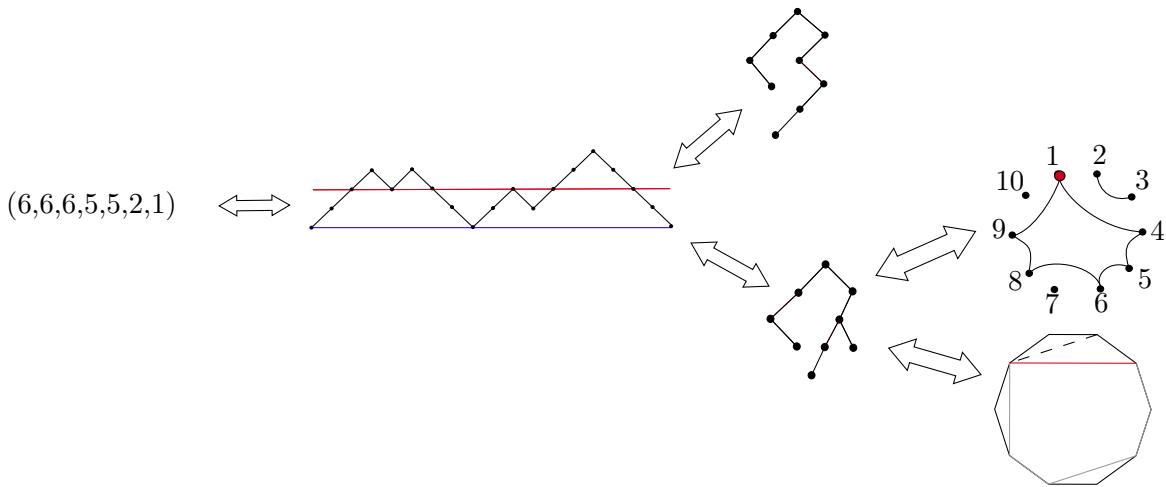

    \centering
%
\caption{A $k$-Naples parking function and its corresponding Catalan objects.}
    \label{fig:bijectionmap}
\end{figure}

\section{Future work}\label{sec:FutureWorks}

As we have said, all rearrangements of parking functions are still parking functions.
Using this fact, it is possible to find simple labelling rules on Dyck paths and trees that correspond to every parking function \cite{YanBook}. 
One area of future research is to explore a way of describing which rearrangements of descending $k$-Naples parking functions are still $k$-Naples parking functions based on a labelling of the objects with which they are in bijection.

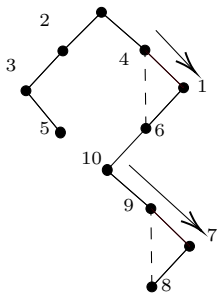
\begin{figure}[h]
    \centering
    
\begin{tikzpicture}[x=0.75pt,y=0.75pt,yscale=-1,xscale=1]

\draw  [fill={rgb, 255:red, 0; green, 0; blue, 0 }  ,fill opacity=1 ] (276.13,85.66) .. controls (276.13,84.29) and (277.24,83.18) .. (278.61,83.18) .. controls (279.98,83.18) and (281.08,84.29) .. (281.08,85.66) .. controls (281.08,87.03) and (279.98,88.14) .. (278.61,88.14) .. controls (277.24,88.14) and (276.13,87.03) .. (276.13,85.66) -- cycle ;
\draw  [fill={rgb, 255:red, 0; green, 0; blue, 0 }  ,fill opacity=1 ] (318.02,124.75) .. controls (318.02,123.38) and (319.13,122.27) .. (320.5,122.27) .. controls (321.87,122.27) and (322.98,123.38) .. (322.98,124.75) .. controls (322.98,126.12) and (321.87,127.23) .. (320.5,127.23) .. controls (319.13,127.23) and (318.02,126.12) .. (318.02,124.75) -- cycle ;
\draw  [fill={rgb, 255:red, 0; green, 0; blue, 0 }  ,fill opacity=1 ] (337.02,104.25) .. controls (337.02,102.88) and (338.13,101.77) .. (339.5,101.77) .. controls (340.87,101.77) and (341.98,102.88) .. (341.98,104.25) .. controls (341.98,105.62) and (340.87,106.73) .. (339.5,106.73) .. controls (338.13,106.73) and (337.02,105.62) .. (337.02,104.25) -- cycle ;
\draw  [fill={rgb, 255:red, 0; green, 0; blue, 0 }  ,fill opacity=1 ] (317.52,85.25) .. controls (317.52,83.88) and (318.63,82.77) .. (320,82.77) .. controls (321.37,82.77) and (322.48,83.88) .. (322.48,85.25) .. controls (322.48,86.62) and (321.37,87.73) .. (320,87.73) .. controls (318.63,87.73) and (317.52,86.62) .. (317.52,85.25) -- cycle ;
\draw [fill={rgb, 255:red, 0; green, 0; blue, 0 }  ,fill opacity=1 ]   (339.5,104.25) -- (320.5,124.75) ;
\draw [fill={rgb, 255:red, 208; green, 2; blue, 27 }  ,fill opacity=1 ]   (320,85.25) -- (339.5,104.25) ;
\draw    (320.5,124.75) -- (301,145.75) ;
\draw  [fill={rgb, 255:red, 0; green, 0; blue, 0 }  ,fill opacity=1 ] (298.52,145.75) .. controls (298.52,144.24) and (299.63,143.01) .. (301,143.01) .. controls (302.37,143.01) and (303.48,144.24) .. (303.48,145.75) .. controls (303.48,147.26) and (302.37,148.49) .. (301,148.49) .. controls (299.63,148.49) and (298.52,147.26) .. (298.52,145.75) -- cycle ;

\draw  [fill={rgb, 255:red, 0; green, 0; blue, 0 }  ,fill opacity=1 ] (274.92,126.84) .. controls (274.92,125.33) and (276.02,124.1) .. (277.39,124.1) .. controls (278.76,124.1) and (279.87,125.33) .. (279.87,126.84) .. controls (279.87,128.36) and (278.76,129.58) .. (277.39,129.58) .. controls (276.02,129.58) and (274.92,128.36) .. (274.92,126.84) -- cycle ;
\draw [fill={rgb, 255:red, 0; green, 0; blue, 0 }  ,fill opacity=1 ]   (298.08,66.15) -- (278.61,85.66) ;
\draw [fill={rgb, 255:red, 0; green, 0; blue, 0 }  ,fill opacity=1 ]   (260,105.75) -- (277.39,126.84) ;
\draw [fill={rgb, 255:red, 0; green, 0; blue, 0 }  ,fill opacity=1 ]   (278.61,85.66) -- (260,105.75) ;
\draw  [fill={rgb, 255:red, 0; green, 0; blue, 0 }  ,fill opacity=1 ] (295.61,66.15) .. controls (295.61,64.78) and (296.72,63.67) .. (298.08,63.67) .. controls (299.45,63.67) and (300.56,64.78) .. (300.56,66.15) .. controls (300.56,67.52) and (299.45,68.63) .. (298.08,68.63) .. controls (296.72,68.63) and (295.61,67.52) .. (295.61,66.15) -- cycle ;
\draw  [fill={rgb, 255:red, 0; green, 0; blue, 0 }  ,fill opacity=1 ] (257.52,105.75) .. controls (257.52,104.38) and (258.63,103.27) .. (260,103.27) .. controls (261.37,103.27) and (262.48,104.38) .. (262.48,105.75) .. controls (262.48,107.12) and (261.37,108.23) .. (260,108.23) .. controls (258.63,108.23) and (257.52,107.12) .. (257.52,105.75) -- cycle ;
\draw [fill={rgb, 255:red, 0; green, 0; blue, 0 }  ,fill opacity=1 ]   (298.08,66.15) -- (320,85.25) ;
\draw    (325.5,75.25) -- (344.13,95.04) ;
\draw [shift={(345.5,96.5)}, rotate = 226.74] [color={rgb, 255:red, 0; green, 0; blue, 0 }  ][line width=0.75]    (10.93,-3.29) .. controls (6.95,-1.4) and (3.31,-0.3) .. (0,0) .. controls (3.31,0.3) and (6.95,1.4) .. (10.93,3.29)   ;
\draw  [dash pattern={on 4.5pt off 4.5pt}]  (320.5,124.75) -- (320,85.25) ;
\draw  [fill={rgb, 255:red, 0; green, 0; blue, 0 }  ,fill opacity=1 ] (321.02,204.75) .. controls (321.02,203.38) and (322.13,202.27) .. (323.5,202.27) .. controls (324.87,202.27) and (325.98,203.38) .. (325.98,204.75) .. controls (325.98,206.12) and (324.87,207.23) .. (323.5,207.23) .. controls (322.13,207.23) and (321.02,206.12) .. (321.02,204.75) -- cycle ;
\draw  [fill={rgb, 255:red, 0; green, 0; blue, 0 }  ,fill opacity=1 ] (340.02,184.25) .. controls (340.02,182.88) and (341.13,181.77) .. (342.5,181.77) .. controls (343.87,181.77) and (344.98,182.88) .. (344.98,184.25) .. controls (344.98,185.62) and (343.87,186.73) .. (342.5,186.73) .. controls (341.13,186.73) and (340.02,185.62) .. (340.02,184.25) -- cycle ;
\draw  [fill={rgb, 255:red, 0; green, 0; blue, 0 }  ,fill opacity=1 ] (320.52,165.25) .. controls (320.52,163.88) and (321.63,162.77) .. (323,162.77) .. controls (324.37,162.77) and (325.48,163.88) .. (325.48,165.25) .. controls (325.48,166.62) and (324.37,167.73) .. (323,167.73) .. controls (321.63,167.73) and (320.52,166.62) .. (320.52,165.25) -- cycle ;
\draw [fill={rgb, 255:red, 0; green, 0; blue, 0 }  ,fill opacity=1 ]   (342.5,184.25) -- (323.5,204.75) ;
\draw [fill={rgb, 255:red, 208; green, 2; blue, 27 }  ,fill opacity=1 ]   (323,165.25) -- (342.5,184.25) ;
\draw [fill={rgb, 255:red, 0; green, 0; blue, 0 }  ,fill opacity=1 ]   (301.08,146.15) -- (323,165.25) ;
\draw    (312,143.25) -- (347.02,175.15) ;
\draw [shift={(348.5,176.5)}, rotate = 222.32999999999998] [color={rgb, 255:red, 0; green, 0; blue, 0 }  ][line width=0.75]    (10.93,-3.29) .. controls (6.95,-1.4) and (3.31,-0.3) .. (0,0) .. controls (3.31,0.3) and (6.95,1.4) .. (10.93,3.29)   ;
\draw  [dash pattern={on 4.5pt off 4.5pt}]  (323.5,204.75) -- (323,165.25) ;

\draw (265.5,66) node [anchor=north west][inner sep=0.75pt]   [align=left] {{\tiny 2}};
\draw (248.5,88.5) node [anchor=north west][inner sep=0.75pt]   [align=left] {{\tiny 3}};
\draw (265.5,120.5) node [anchor=north west][inner sep=0.75pt]   [align=left] {{\tiny 5}};
\draw (305.5,85.5) node [anchor=north west][inner sep=0.75pt]   [align=left] {{\tiny 4}};
\draw (345,99) node [anchor=north west][inner sep=0.75pt]   [align=left] {{\tiny 1}};
\draw (323.5,121.5) node [anchor=north west][inner sep=0.75pt]   [align=left] {{\tiny 6}};
\draw (286.5,135.5) node [anchor=north west][inner sep=0.75pt]   [align=left] {{\tiny 10}};
\draw (308,159) node [anchor=north west][inner sep=0.75pt]   [align=left] {{\tiny 9}};
\draw (350,174.5) node [anchor=north west][inner sep=0.75pt]   [align=left] {{\tiny 7}};
\draw (327,199.5) node [anchor=north west][inner sep=0.75pt]   [align=left] {{\tiny 8}};

\end{tikzpicture}

    \caption{Labelling binary trees for $k=1$}
    \label{fig:TreeLabbel}
\end{figure}

When $k=1$, we noticed some patterns when labelling trees that correspond to a $1$-Naples parking function as illustrated in Figure \ref{fig:TreeLabbel}. 
Note that in these labellings, the root is unlabeled as it does not correspond to a car, and a node must have a lower labelling then its right child. 
Now, for the patterns, if a direct right descendent of the root does not have a left child, then it must have a higher labelling then its right child. 
Once this ends, we look at the final direct right descendent in the section without a left child. 
We already know it must have a higher label than its right child, so now look at the right child that is connected to the original node in Figure \ref{fig:TreeLabbel} by a dotted line.
If the grandchild has a smaller than the original, there is no problem here with the rearrangement. Otherwise, the process must begin again.
We can see this is a rather convoluted process, and neither of the rules follow in a satisfying way once $k>1$. 

A similar area of study is to find what labelling conventions correspond to ascending $k$-Naples parking functions using our other bijections. 
We have seen the result for both Dyck paths and binary trees, but we have not studied the condition on dissections or rooted non-crossing partitions.
Perhaps on these the condition is relatively nice and could help us understand rearrangements more.

Lastly, there are many objects counted by the Catalan numbers and their convolutions that we have not discussed here. 
Finding and understanding more bijections could help us better understand the structure of $k$-Naples parking functions, and some objects may be better suited for describing rearrangements. 
One could also look for bijections for ascending strictly $k$-Naples parking functions, ascending or descending parking preferences that are not $k$-Naples parking functions, or descending $k$-Naples parking functions whose ascending rearrangements are not $k$-Naples parking functions.

\section*{Acknowledgements}
Part of this research was performed with support from the Institute for Pure and Applied Mathematics (IPAM), which is supported by the National Science Foundation (Grant No.\ DMS-1440415).
PEH was supported through a Karen EDGE Fellowship. 
ARVM was partially supported by the National Science Foundation under Awards DGE-1247392, KY-WV LSAMP Bridge to Doctorate HRD-2004710, and DMS-2102921.

\bibliography{bibliography}

\end{document}